\documentclass[12pt, twoside]{article}
\usepackage{amsmath,amsthm,amssymb}
\usepackage{times}
\usepackage{enumerate}

\def\titlerunning#1{\gdef\titrun{#1}}
\makeatletter
\def\author#1{\gdef\autrun{\def\and{\unskip, }#1}\gdef\@author{#1}}
\def\address#1{{\def\and{\\\hspace*{18pt}}\renewcommand{\thefootnote}{}%
\footnote {#1}}%
\markboth{\autrun}{\titrun}}
\makeatother
\def\email#1{e-mail: #1}
\def\subjclass#1{{\renewcommand{\thefootnote}{}%
\footnote{\emph{Mathematics Subject Classification (2010):} #1}}}
\def\keywords#1{\par\medskip
\noindent\textbf{Keywords.} #1}

\newcommand{\Hmm}[1]{\leavevmode{\marginpar{\tiny%
$\hbox to 0mm{\hspace*{-0.5mm}$\leftarrow$\hss}%
\vcenter{\vrule depth 0.1mm height 0.1mm width \the\marginparwidth}%
\hbox to
0mm{\hss$\rightarrow$\hspace*{-0.5mm}}$\\\relax\raggedright #1}}}

\newtheorem{thm}{Theorem}[section]
\newtheorem{cor}[thm]{Corollary}
\newtheorem{lem}[thm]{Lemma}

\newtheorem{pro}[thm]{Proposition}

\theoremstyle{definition}
\newtheorem{defin}[thm]{Definition}
\newtheorem{rem}[thm]{Remark}

\numberwithin{equation}{section}

\newcommand{\R}{{\mathbb R}}

\newcommand{\N}{{\mathbb N}}

\let\H\undefined
\let\L\undefined
\newcommand{\H}{H}
\newcommand{\L}{L}
\newcommand{\F}{F}

\newcommand{\ph}{{\varphi}}
\newcommand{\eps}{{\varepsilon}}

\newcommand{\si}{{\sigma}}

\newcommand{\lm}{{\lambda}}

\newcommand{\ov}[1]{\overline{ #1}}
\newcommand{\ow}[1]{\widetilde{ #1}}

\newcommand{\dt}{\,\mathrm{d}t}

\def\ga{\alpha}     \def\gb{\beta}

\def\gf{\phi}       \def\vgf{\varphi}    
            \def\gl{\lambda}

\begin{document}

\baselineskip=17pt

\titlerunning{Criticality theory   on graphs}
\title{Criticality theory  for Schr\"odinger operators\\ on graphs}
\author{Matthias Keller
\and
Yehuda Pinchover
\and
Felix Pogorzelski}

\date{}

\maketitle

\address{M.~Keller:  Institut f\"ur Mathematik, Universit\"at Potsdam, 14476  Potsdam, Germany;
\email{mkeller@math.uni-potsdam.de}
\and
Y.~Pinchover and F.~Pogorzelski: Department of Mathematics, Technion-Israel Institute of Technology, 32000 Haifa, Israel;
\email{pincho@technion.ac.il and felixp@technion.ac.il}}
%
%

\begin{abstract}
	We study Schr\"odinger operators given by positive quadratic forms on infinite graphs.
From there, we develop a criticality theory
for Schr\"odinger operators on general weighted graphs.

\end{abstract}
\section{Introduction}

Schr\"odinger operators are an important class of operators in analysis and mathematical physics.  One of the first and most fundamental questions in the analysis of positive Schr\"odinger operators is the question of criticality. In fact, it is the starting point to study Liouville theorems, the large time behaviour of the heat kernel, Hardy inequalities, properties of the ground state as well as basic questions in spectral theory.

Such a theory is classical for second-order linear elliptic operators (not necessarily symmetric) with real coefficients for which we refer the
reader to \cite{M86,P07,Pins95} and references therein. See also \cite{PP} (and references therein)
for recent developments in the quasilinear case,  and \cite{Tak14,Tak16}
for the case of generalized Schr\"odinger forms.

In this paper we present a criticality theory
for positive Schr\"odinger operators on general weighted graphs. First applications of this theory were already obtained in \cite{KePiPo2} and \cite{BP}.

While criticality theory was so far not studied for graphs, there is a closely related phenomena called recurrence. This notion appears in the setting of random walks, the situation of positive matrices and in the very general context of Dirichlet forms. In each of these settings there exists a vast body of literature so we mention here only  the monograph \cite{Woe-Book} for random walks, the chapter in the survey   \cite[Section~6]{MoharWoess89} for positive matrices and \cite{Fuk} for the theory of Dirichlet forms.
We elaborate in  depth about the relationship of recurrence and criticality in Remark~\ref{r:4.5} in Section~\ref{s:char_crit}.

The paper is structured as follows. In the following section the basic
setting is introduced. In Section~\ref{s:SA_operators} we discuss self-adjoint realizations of the formal Laplacian.
This is used to prove an Allegretto-Piepenbrink
theorem in Section~\ref{s:AP} which utilizes a local Harnack inequality and a ground state transform. We proceed in Section~\ref{s:char_crit} by a characterization of criticality and subcriticality. This includes in Section~\ref{s:GF} a particular discussion of the Green function. We continue to characterize a phenomenon called uniform subcriticality in Section~\ref{s:USC}. Finally, we characterize a phenomenon called {\em null-criticality} in Section~\ref{s:NC}. In particular, we characterize it in terms of the large time behavior of the heat kernel and in terms of the behavior of the Green function near criticality.
\section{Set up}
\subsection{Graphs}
Let $X$ be an infinite set equipped with the discrete
topology. A graph over $X$ is a symmetric function $b:X\times
X\to[0,\infty)$ with zero diagonal such that it is \emph{locally summable}, that is,
$$\sum_{y\in  X}b(x, y)<\infty \quad \textup{ for all } x\in X.$$
We call the elements of $X$  \emph{vertices}. We say that $x,y\in X$ are
\emph{adjacent} or \emph{neighbors} or \emph{connected by an edge}
if $b(x,y)>0$ in which case we write $x\sim y$.
We call  $ b $ \emph{connected} if for every $x$ and $y$ in $X$
there are $x_{0},\ldots,x_{n}$ in $X$ such that $x_{0}=x$, $x_{n}=y$
and $x_{i}\sim x_{i+1}$ for $i=0,\ldots,n-1$.
\subsection{Formal Schr\"odinger operators and forms}\label{Subsection-Formal}
Let $W\subseteq X$, we denote by $C(W)$ (resp., $C_c(W)$) the space
of real valued functions on $W$ (resp., with compact support in
$W$). By extending functions by zero on $X\setminus W$ the space
$C(W)$  will be considered as a subspace of  $C(X)$.

We write $f\ge c$ (resp., $f=c$) in $W$, whenever a
function $f\in C(W)$ is  larger or equal (resp., equal)
to the function which takes the constant value $c\in\R$ on $W$. In
particular, with a slight abuse of notation, we do not distinguish
between constants and constant functions in notation and may write
for example $1$ for the function that takes constantly the value 1
on $X$.  We
say that $f\in C(W)$ is \emph{positive} in $W$ if  $f\geq 0$ and $f\neq 0$ in
$W$, in this case, we also use the notation $f\gneqq 0$.
We use the notation $f=f_{+}-f_{-}$, where $f_{\pm}:=(0\vee \pm\,
f)$ are the positive and the negative parts of $f$.


Given a  graph $b$ over $X$, we introduce the associated \emph{formal
Laplacian} ${\L}={\L}_{b}$ acting on the space
\begin{align*}
{\F}(X):=\{f\in C(X) \mid \sum_{y\in X}b(x,y)|f(y)|<\infty\mbox{ for
all } x\in X\},
\end{align*}
by
\begin{align*}
{\L} f(x):=\sum_{y\in X}b(x,y)(f(x)-f(y)).
\end{align*}
By the summability assumption on $b$ we have $\ell^\infty(X) \subset
F(X)$.


For a potential $q:X\to\mathbb{R}$, we define the \emph{formal
Schr\"odinger operator} ${\H}$ on $\F(X)$ by
\begin{align*}
{\H}:={\L}+q.
\end{align*}
The associated \emph{bilinear form} $h$ of $H$ on $C_{c}(X)\times C_{c}(X)$ is given by
\begin{align*}
h(\ph,\psi)\!:=\!\frac{1}{2}\sum_{x,y\in
X}\!\!\!b(x,y)(\ph(x)\!-\!\ph(y))(\psi(x)\!-\!\psi(y)) \!+\!\sum_{x\in
X}\!q(x)\ph(x)\psi(x).
\end{align*}
We denote by $h(\ph):=h(\ph,\ph)$ the induced quadratic form on $C_c(X)$. Furthermore, we write $h\ge0\mbox{ on }C_{c}(X)$ (or in short $h\geq 0$)
if $h(\ph)\ge0$ for all $\ph\in C_{c}(X)$.

\medskip

A strictly positive function $m:X\to(0,\infty)$ extends to a measure
via $m(A):=\sum_{x\in A}m(x)$, where $A\subseteq X$. The real Hilbert space $\ell^{2}(X,m)$ is the space of all
$m$-square summable functions, equipped with the  scalar product
\begin{align*}
\langle{f},g\rangle_{m}:=\sum_{X}f g m:=\sum_{x\in X}f(x)g(x)m(x)
\qquad f,g \in \ell^{2}(X,m),
\end{align*}
 and the induced norm $\|\cdot\|=\|\cdot\|_{m}$.  Given a measure $m$ we speak of a graph $b$ over $(X,m)$.

With a slight abuse of notation, we keep writing $\ell^{2}(X,m)$ also in the case where $m\gneqq 0$.
 Of a particular importance is
the case when $m=1$ is the counting measure. In this case, we denote by
$\ell^{2}(X)$ the Hilbert space of square summable functions equipped with the
scalar product
\begin{align*}
\langle{f},g\rangle:=\langle{f},g\rangle_1=\sum_{X}f g=\sum_{x\in X}f(x)g(x) \qquad f,g \in \ell^{2}(X).
\end{align*}

There is a Green formula relating $H$ and $h$. The formula follows
by a direct algebraic manipulation, where one has to ensure that all
the involved sums converge absolutely. This is however a consequence of
the Cauchy-Schwarz inequality and Fubini's theorem. For details see
\cite[Lemma 4.7]{HK}.

\begin{lem}[Green formula]\label{l:Green}
	For all $f\in {\F}(X)$ and $\psi\in C_{c}(X)$ one has
	\begin{align*}
	\frac{1}{2}\sum_{x,y\in X}b(x,y)(f(x)-f(y))(\psi(x)-\psi(y))
	+\sum_{x\in
		X}q(x)f(x)\psi(x)\\
	=\sum_{x\in X}({\H}f)(x)\psi(x)=\sum_{x\in X}f(x)({\H}\psi)(x).
	\end{align*}
	Furthermore,  $\H C_{c}(X)\subseteq \ell^{2}(X)$, so if $\ph,\psi\in
	C_{c}(X)$, then
	\begin{align*}
	h(\ph,\psi)=\langle{{\H}\ph},{\psi}\rangle=\langle{\ph},{\H}{\psi}\rangle.
	\end{align*}
	Hence, we can recover $ H $ from $ h $, so we also speak of $ H $ associated to $ h. $
\end{lem}

\section{Self-adjoint realization of the operator}\label{s:SA_operators}
In this section we discuss the closability  of the form $h$ in weighted
$\ell^{2}$-spaces over $X$. This allows us to define a self-adjoint
operator which coincides with $\H$ in the case of the counting
measure. This opens the door to the use of operator theoretic
arguments in the succeeding discussion.

The next theorem states that we can always close a {\em nonnegative} form $h$
on $\ell^{2}(X,m)$ for every measure $m$, and that the corresponding
selfadjoint operator $H^{(m)}$ acts as $\frac{1}{m}\H$ on functions
in $\F\cap D(H^{(m)})$, where $D(H^{(m)})$ is the domain  of $H^{(m)}$.

We recall that a form $h\ge0$ on $C_{c}(X)$
	is {\em closable} in $\ell^{2}(X,m)$ if a Cauchy
	sequence $(\ph_{n})$ in $C_{c}(X)$ with respect to the form norm $\|\cdot\|_{h,m}:={(h(\cdot)+\|\cdot\|_{m}^{2})^{1/2}}$ such that  $\|\ph_{n}\|_{m}\to 0 $
	satisfies $h(\ph_{n})\to0$, as $n\to\infty$. The closure  $h^{(m)}$
	is then the extension of $h$ to $\ov{C_{c}(X)}^{\|\cdot\|_{h,m}}$.
	
	Recall also that $ D(H^{(m)}) $, the domain  of the corresponding
	selfadjoint operator $H^{(m)}$, is the linear space of  all functions $f\in D(h^{(m)}) $ such that there is $ g \in \ell^{2}(X,m)$ so that $ h^{(m)}(f,h)=\langle g,h\rangle_{m} $ for all $ h\in D(h^{(m)}) $.
\begin{thm}\label{t:closable} Let $b$ be a graph over $X$ and  $q$ be a potential such that
	$h\ge 0$ on $ C_{c}(X)$. Then, $h$ is closable in $\ell^{2}(X,m)$
	for each measure $m:X\to(0,\infty)$. The closure $ h^{(m)} $ of $ h $  satisfies the first Beurling-Deny criterium, i.e., for all $ f\in D(h^{(m)}) $ we have $ |f|\in D(h^{(m)}) $ and
	\begin{align*}
	h^{(m)}(|f|)\leq h^{(m)}(f).
	\end{align*}
	Moreover, the selfadjoint operator
	$H^{(m)}$ associated  with $ h^{(m)}$ satisfies for
	all $f\in D(H^{(m)})\cap \F(X)$
	$$H^{(m)}f=\frac{1}{m}{\H}f.$$
\end{thm}
\begin{proof}
	The form $h$ is closable in any $\ell^{2}(X,m)$ by
	\cite[Theorem~2.9]{GKS}. We denote the closure by $ h^{(m)}. $
	
	We show next that the inequality
	\begin{align*}
	h(|\ph|)\leq h(\ph),\qquad \ph\in C_{c}(X),
	\end{align*}
	extends to functions $ f\in D(h^{(m)}) $.
	Let  $ f\in D(h^{(m)}) $ and let $ (\ph_{n}) $ be a sequence in $ C_{c}(X) $ that converges to $ f $ with respect to with respect to the form norm
	$$\|\cdot\|_{h^{(m)}}:=  \sqrt{h^{(m)}(\cdot)+\|\cdot\|^{2}_m}\;.$$  Then,
	\begin{align*}
	h^{(m)}(|\ph_{n}|)=h(|\ph_{n}|)\leq h(\ph_{n})=h^{(m)}(\ph_{n})\to h^{(m)}(f).
	\end{align*}
	Hence, the sequence $ (h^{(m)}(|\ph_{n}|)) $ is bounded and $ (|\ph_{n}|) $ converges to $ |f| $ in $ \ell^{2}(X,m) $ and, therefore, also pointwise. Since $ h^{(m)} $ is closed, it is lower semicontinuous. Hence,
	\begin{align*}
	h^{(m)}(|f|)\leq \liminf_{n\to\infty}h^{(m)}(|\ph_{n}|)<\infty,
	\end{align*}
	which yields $ |f|\in D(h^{(m)}) $ and the inequality
	\begin{align*}
	h^{(m)}(|f|)\leq h^{(m)}(f).
	\end{align*}
	
	Now, let $H^{(m)}$ be the selfadjoint  operator
	associated to the closure $h^{(m)}$ in $\ell^{2}(X,m)$. Let $f\in
	D(H^{(m)})\cap \F(X)$.
	Using Lemma~\ref{l:wlogmonotone} below, assume that there exist  $f_{n}\in C_{c}(X)$  such that $|f_{n}|\leq |f|$,
	and $(f_{n})$ converges to $f$ with respect to the form norm $\|\cdot\|_{h^{(m)}}$.
	Then,   by the Green formula we have with $\delta_{x}:=1_{\{x\}}/m(x)$
	\begin{align*}
	h^{(m)}(f_{n},\delta_{x})=
	h(f_{n},\delta_{x})=\langle{{\H}f_{n}},{\delta_{x}}\rangle_{1}
	=\frac{1}{m(x)}{\H}f_{n}(x).
	\end{align*}
	Again, we use the fact that $\ell^{2}$-convergence on a discrete space implies pointwise
	convergence. Thus,  by Lebesgue's dominated convergence theorem,
	with $f\in \F(X)$ and $|f_{n}|\leq|f|$, we get
	\begin{align*}
	\lim_{n\to\infty}\frac{1}{m}{\H}f_{n}(x)
	&    \!=\!\frac{f(x)}{m(x)}\!\!\left({\sum_{y\in X}b(x,y)+q(x)}\!\right)-
	\lim_{n\to\infty}\!\!\left(\!\!\frac{1}{m(x)}\sum_{y\in X}b(x,y)f_{n}(y)\!\right)\\
	&   =\frac{1}{m(x)}{\H}f(x). \qedhere
	\end{align*}
\end{proof}
\begin{lem} \label{l:wlogmonotone}
Let $b$ be a graph over $X$ and assume that $\ov h$ is a quadratic form which is closed
in $\ell^2(X,m)$ for some measure $m:X \to (0,\infty)$. Then, for every $f \in D(\ov h) \cap F(X)$,
there exists a sequence $(f_n)$ in $C_c(X)$ such that $|f_n| \leq |f|$ for all $n \in \mathbb{N}$
and $f_n \to f$ in $\|\cdot\|_{\ov h}$.
\end{lem}

\begin{proof}
	By definition of a closed form $\ov h$, there exists
	a sequence $(g_{n})\subset  C_{c}(X)$ such that $g_{n}\to f$ with respect
	$\|\cdot\|_{\ov h}$. By decomposing $f$ and $g_{n}$ into positive
	and negative part, it suffices to consider  $f\ge 0$ and $g_{n}\ge
	0$. Set $f_{n}:=g_{n}\wedge f=(f+g_{n}-|f-g_{n}|)/2$. Since
	$\ov h(|\ph|)\leq \ov h(\ph)$ for all $\ph\in D(\ov h)$, we
	obtain
	\begin{align*}
	\ov h&(f-f_{n})
	=  \frac{1}{4}\left({
		\ov h(f-g_{n})+2\ov h(f-g_{n},|f-g_{n}|) + \ov h(|f-g_{n}|)}\right)
	\\
	&\leq     \frac{1}{4}\left({ \ov h(f-g_{n})+2\ov h(f-g_{n})^{\frac{1}{2}}\ov h(|f-g_{n}|)^{\frac{1}{2}}
		+ \ov h(|f-g_{n}|)}\right)
	\leq \ov h(f-g_{n}).
	\end{align*}
	Similarly, $\|f_n-f\|_m\leq \|g_{n}-f\|_m$. Hence, $\|f_n-f\|_{\ov h}\leq \|g_{n}-f\|_{\ov h}$. Thus, the positive
	and negative parts of $f_n$ converge to $f_\pm$, and this finishes the proof.
\end{proof}

\begin{rem}
	One can characterize all positive closed forms on discrete sets that  have $ C_{c}(X) $ as a core and that satisfy the second Beurling-Deny criterion, i.e.,  $ h(f\vee 1)\leq h(f) $ for all $ f $ in the domain of $ h $, see \cite{KL1}. Those forms are exactly closures of a form $ h $ on $ C_{c}(X) $ as above associated with a graph $ b $  and a potential $ q\ge0 $. So, it is natural to ask the question whether our setting actually covers all positive closed forms that satisfy the first Beurling-Deny criterion
 $$  h(|\ph|)\leq h(\ph),\qquad \ph\in C_{c}(X).  $$
and have $ C_{c}(X) $ as a core. This is however not the case as the following example shows:
	
	Let $X=\N_{0}  $ and let $ h $ be a form on $ C_{c}(X) $ acting as
		\begin{align*}
		h(\ph)=\sum_{k\in \N_{0}} \ph(k)^{2}-\sum_{k\in \N}\frac{1}{k}\ph(0)\ph(k),\qquad\ph\in C_{c}(X).
		\end{align*}
	It is not hard to see that  $ h\ge0 $: If $ \ph(0)=0 $, then clearly $ h(\ph)\ge0 $, so assume $ \ph(0)=1 $. Then $ h(\ph)=1+\sum_{k\in\N}\ph(k)(\ph({k})-1/k) $. So minimizing each term in the sum  $ \ph(k)(\ph(k)-1/k) $ yields that the minimum is assumed for the function $ \ph_{0} $ such that $ \ph_{0}(k) =1/2k$, $ k\ge1 $ and $ \ph_{0}(0)=1 $. However, $ h(\ph_{0})=1-\sum_{k\in\N}1/(2k)^{2} =1-\pi^{2}/24\ge0$.
	  Furthermore,
	an immediate calculation yields that the associated bilinear form
	\begin{align*}
	 h(\ph,\psi)=\sum_{k\in \N_{0}} \ph(k)\psi(k)-\frac{1}{2}\sum_{k\in \N}\frac{1}{k}(\ph(0)\psi(k)+\ph(k)\psi(0)),\qquad\ph\in C_{c}(X).
	\end{align*}
	satisfies $ h(1_{k},1_{l})\leq 0 $, $ k\neq l $. This, however, can  easily be seen to be equivalent to the first Beurling-Deny criterion.

 	Moreover, taking the counting measure one easily sees that every $ \ell^{2}(\N_{0}) $-sequence $ (\ph_{n}) $ converging to zero  satisfies  $ h(\ph_{n})\to0 $
	since $| h( \ph)|\leq (1+\pi^{2}/6)\|\ph\|^{2} $ for all $ \ph\in C_{c}(X) $.
	Hence, $ h $  is closable in $ \ell^{2}(\N_{0}) $. On the other hand, it is also obvious that one can not write $ h $ as a form associated to a graph $ b $ and a potential $ q $ as above (indeed, one would get the difference of two divergent sums).
\end{rem}

Returning to the operator $H^{(m)}$ from above, we denote its spectrum
by $\sigma(H^{(m)})$. The well known Laplace transform of $ H^{(m)} $ is used to related the resolvent and the semigroup.

\begin{lem}[Laplace transform]\label{l:Laplace} Let $b$ be a  graph over $X$, let $q$ be a potential such that $h\ge0$ on $C_{c}(X)$ and let $ m $ be a measure of full support on $ X $.
	For all  $ \lm<\inf\si(H^{(m)}) $ on $ \ell^{2}(X,m) $ and $ f\in\ell^{2}(X,m) $
	\begin{align*}
	(H^{(m)}-\lm)^{-1}f=\int_{0}^{\infty}\mathrm{e}^{t\lm}\mathrm{e}^{-tH^{(m)}} f\dt.
	\end{align*}
\end{lem}
\begin{proof}This follows from the spectral theorem.
\end{proof}

With the help of this lemma we remark the following fact for the resolvent of $ H^{(m)} $.

\begin{cor}\label{c:posivitity_preserving}
	Let $b$ be a graph over $X$ and  $q$ be a potential such that
	$h\ge 0$ on $ C_{c}(X)$. Then, for any measure $ m:X\to(0,\infty) $  the resolvent of the operator $ H^{(m)}  $ is positivity preserving, i.e. for any $ f\in \ell^{2}(X,m) $ with $ f\ge0 $ we have $ (H^{(m)}-\lm)^{-1}f\ge0  $ for all $ \lm<\inf\si(H^{(m)}) $.
\end{cor}
\begin{proof} Positivity preservation of the semigroup $ \mathrm{e}^{-tH^{(m)}} $ follows from the first Beurling-Deny criterium proven above, see \cite[Theorem~XIII.50]{RS4}. Hence, positivity preservation of the resolvent follows from the Laplace transform.
\end{proof}
To study the spectrum of the operator $H^{(m)}$ via solutions it is
essential to know the action of $H^{(m)}$ explicitly.  The theorem
above (Theorem~\ref{t:closable}) only guarantees this knowledge for
functions in $D(H^{(m)})\cap \F(X)$.

The following proposition provides us with two criteria to ensure
that $D(H^{(m)})\subseteq \F(X)$. The first is on the Laplacian part
${\L}$ of ${\H}$ and the second on the  potential $q$. Morally, the
assumption on ${\L}$ implies $\ell^{2}\subseteq \F(X)$ and the
assumption on $q$ implies that the functions of finite energy are
included in $\F(X)$. The statements can all be extracted from the
considerations in \cite{GKS} and \cite{KL2}. However, we include a
short proof for the convenience of the reader.

\begin{pro}\label{p:H} Let a graph $b$ over a discrete measure space
	$(X,m)$ be given, and let $q:X\to\R$ be such that $h\ge0$ on
	$C_{c}(X)$. Then, $D(H^{(m)})\subseteq \F(X)$ if one of the
	following assumptions is satisfied:
	\begin{itemize}
		\item[(a)] $\frac{1}{m}{\L}C_{c}(X)\subseteq \ell^{2}(X,m)$. This condition is
		in particular satisfied if one of the following assumptions
		holds:
		\begin{itemize}
			\item[(a1)] The graph $b$ is locally finite.
			\item[(a2)] The measure $m$ satisfies $\inf_{x\in X} m(x)>0$.
			\item[(a3)] The function
			$\mathrm{Deg}_{m}(x):=\frac{1}{m(x)}\sum_{y\in X}b(x,y)$ is bounded on $X$.
		\end{itemize}
		\item[(b)] For some $\eps>0$ we have $h(\ph)\ge \eps\langle{{q_{-}}\ph},{\ph}\rangle_{1}$ for all $  \ph\in C_{c}(X)$.
	\end{itemize}
\end{pro}
\begin{proof}
	(a) By definition of $H^{(m)}$ we have $D(H^{(m)})\subseteq
	\ell^{2}(X,m)$, and we show $\ell^{2}(X,m)\subseteq \F(X)$ under the
	assumption (a). By considering characteristic functions of vertices
	one finds that $\frac{1}{m}{\L}C_{c}(X)\subseteq \ell^{2}(X,m)$ is
	equivalent to the fact that the functions
	$\ph_{x}=b(x,\cdot)/m(\cdot)$, $x\in X$, are in $\ell^{2}(X,m)$.
	Hence, for $f\in \ell^{2}(X,m)$, we have by the Cauchy-Schwarz
	inequality
	\begin{align*}
	\sum_{y\in X}b(x,y)|f(y)|=\langle{\ph_{x}},{|f|}\rangle_{m}
	\leq \|\ph_{x}\|_{m} \|f\|_{m}\,.
	\end{align*}
	Thus, $f\in \F(X)$ and therefore, $D(H^{(m)})\subseteq
	\ell^{2}(X,m)\subseteq \F(X)$. The ``in particular'' statements
	follow easily from the characterization of
	$\frac{1}{m}{\L}C_{c}(X)\subseteq \ell^{2}(X,m)$ by $\ph_{x}\in
	\ell^{2}(X,m)$, $x\in X$.
	
	Assume
	$h_{+}(\ph):=\langle{({\L}+q_{+})\ph},{\ph}\rangle_{1}\ge (1+\eps)
	\langle{{q_{-}}\ph},{\ph}\rangle_{1},$ for all $\ph\in C_{c}(X) $.
	By standard perturbation theory,  the domain $D(h_{+,m})$ of the
	closure $h_{+,m}=\ov h_{+}$ of the quadratic form $h_{+}$ on $\ell^{2}(X,m)$ is equal
	to the domain $D(h^{(m)})$ of $h^{(m)}=\ov h$. (Recall that $h$ is
	closable in $\ell^{2}(X,m)$ by Theorem~\ref{t:closable}.) From standard
	arguments using Fatou's lemma it follows (cf. \cite{KL2}) that
	\begin{align*}
	h_{+,m}(f,f)=\frac{1}{2}\sum_{x,y\in
		X}b(x,y)(f(x)-f(y))^{2}+\sum_{x\in X}q_{+}(x)f(x)^{2}
	\end{align*}
	for $f\in D(h_{+,m})$, and therefore, $$D(h_{+,m})\!\subseteq\! \ow
	D\!:=\!\!\Big\{f\!\in\! C(X)\!\mid \! \ow {h}
	(f)\!:=\!\frac{1}{2}\!\!\sum_{x,y\in
		X}\!b(x,y)(f(x)-f(y))^{2}<\infty\!\Big\}.$$ So, one is left to
	check $\ow D \subseteq \F(X)$ which can be seen via the inequalities
	\begin{align*}
	{\sum_{y\in X}b(x,y)|f(y)|}&\leq {\sum_{y\in X}b(x,y)|f(x)-f(y)|}+
	|f(x)|B(x)\\
	&\leq \sqrt{2}B(x)^{1/2} \ow{h}(f)^{1/2}+
	|f(x)|B(x),
	\end{align*}
	where $B(x):=\sum_{y\in X}b(x,y)<\infty$ by our assumption.
\end{proof}
\begin{rem} Condition (a) of Proposition~\ref{p:H}
	is equivalent to the condition $C_{c}(X)\subseteq D(H^{(m)})$ which can be
	directly read from the Green formula, Lemma~\ref{l:Green}, and the abstract  definition of $D(H^{(m)})$. Furthermore, (a2) implies that
	$$D(H^{(m)})=\{f\in \ell^{2}(X,m)\mid Hf\in \ell^{2}(X,m)\},$$ whenever
	$q\ge0$, see \cite{KL2}. Finally, (a3) implies that the operator
	$H^{(m)}$ is bounded whenever $ q $ is bounded, see \cite[Theorem 9.3]{HKLW}.
\end{rem}
We next draw a corollary from Corollary~\ref{c:posivitity_preserving} and Proposition~\ref{p:H}.
\begin{cor}[Action of $H^{(m)}$]\label{c:H} Let a graph $b$ over a discrete measure space
	$(X,m)$ be given, and let $q:X\to\R$ be such that $h\ge0$. Assume
	one of the conditions in Proposition~\ref{p:H} holds. Then, for all
	$f\in D(H^{(m)})$
	\begin{align*}
	H^{(m)}f=\frac{1}{m}Hf.
	\end{align*}
	In particular, for  $\lm<0$ and $g\in \ell^{2}(X,m)$, the function $u_{\lm}:=(H^{(m)}-\lm)^{-1}g$ is in $F(X)$.  Moreover, if $g\geq 0$, then $u_{\lm}$ is a positive solution of the equation
	\begin{align*}
	(H-\lm m)u_{\lm}=g \qquad \mbox{in } X.
	\end{align*}
\end{cor}

\section{Allegretto-Piepenbrink theorem}\label{s:AP}
Let $\lm_{0}(H^{(m)})$ and $\lm_{0}^{\mathrm{ess}}(H^{(m)})$ be the bottom of
the spectrum  and the bottom of the essential spectrum
of the selfadjoint operator $H^{(m)}$.  In this subsection we prove an Allegretto-Piepenbrink-type theorem relating $\lm_{0}(H^{(m)})$ and $\lm_{0}^{\mathrm{ess}}(H^{(m)})$  to
the existence of positive (super)solutions of the equation $({\H}-\lm m)u=0$.

For $q\!\ge\!0$ and $\lm_{0}(H^{(m)})$ such a result is found in \cite{HK},
and for $\lm_{0}^{\mathrm{ess}}(H^{(m)})$ see \cite{K5}. For
$q =0$ and locally finite graphs see  \cite{BG}, and for
Dirichlet forms see \cite{FLW,LSV}.

\begin{defin}[(Super)harmonic function]
	We say that a function $u$ is $H$-\emph{(super)harmonic} on
	$W\subseteq X$ if $u\in {\F}(X)$ and ${\H}u=0$, ($Hu\ge0$) on $W$.
	We write
	\begin{align*}
	{\H}\ge0 \quad \mbox{on } W
	\end{align*}
	if there exists a positive $H$-superharmonic function $u$ on $W$.
\end{defin}

\begin{thm}[Allegretto-Piepenbrink-type theorem]\label{t:AP}
Let $b$ be a connected infinite graph over $X$ and $q$ be a potential
such that $h\geq 0$ on $C_c(X)$. Let $m$ be a given measure $m:X\to(0,\infty)$ over
$X$.
\begin{itemize}
    \item [(a)] We have
    \begin{align*}
    \lm_{0}(H^{(m)})\ge\sup\{\lm\in \mathbb{R}\mid ({\H}-\lm m)\ge0 \mbox{ on } X\}.
    \end{align*}
    \item [(b)] In addition, if $D(H^{(m)})\subseteq \F(X)$, then
    \begin{align*}
    \lm_{0}(H^{(m)})&=\max\{\lm\in \mathbb{R}\mid ({\H}-\lm m)\ge0 \mbox{ on } X\}.
\end{align*}
    \item [(c)]
    Moreover, if  $\frac{1}{m}{\L}C_{c}(X)\subseteq \ell^{2}(X,m)$, then
    \begin{align*}
    \lm_{0}^{\mathrm{ess}}(H^{(m)})\!
    =\!\sup\big\{\lm\in \mathbb{R}\!\mid\! ({\H}\!-\!\lm m)\!\ge \!0
    \mbox{ on } X\setminus K \mbox{ for a finite } K\subset X\big\}.
    \end{align*}
\end{itemize}
\end{thm}
\begin{rem}
	Clearly, on finite graphs there are no positive $ (\H-\lm) $-harmonic functions  for $ \lm<\lm_{0}(H^{(m)}) $ but only $ (\H-\lm) $-superharmonic functions.
	In \cite{HK} an example is given that shows that in the non locally finite setting there might be only $ (\H-\lm) $-superharmonic functions  for $ \lm<\lm_{0}(H^{(m)}) $ either.
\end{rem}

\begin{rem}
We  formulated Theorem~\ref{t:AP} for the case $h\ge0$. By fixing the measure $m$ before choosing the
potential $q$, one can allow for a potential $q$ such the resulting
form on $C_{c}(X)$ is form bounded from below in $\ell^{2}(X,m)$.
Such forms are also closable (cf.  \cite[Theorem 2.9]{GKS}) and the
arguments carry over verbatim with the obvious adaption made.
\end{rem}

The proof of the Allegretto-Piepenbrink theorem uses a local Harnack inequality and a ground state transform which are proven next.

\subsection{A local Harnack inequality}
The following local Harnack inequality is  a slightly more general
formulation of \cite[Proposition 3.4]{HK} (for earlier versions see
e.g.\@ \cite{Do84,Woj1}). However, the proof can be carried over
verbatim, but as it is short we give it here for the sake of our paper being self-contained.

\begin{lem}[Harnack inequality]\label{l:localHarnack}
	Let $W\subseteq X$ be a connected and finite set and let $f\in
	C(X)$. There exists a positive constant $C=C(H,W,f)$ such that for
	any nonnegative function $u\in \F(X)$  satisfying $(\H-f) u \ge 0$
	on $W$, the following inequality holds
	\begin{align*}
	\max_{W}u\leq C\min_{W}u.
	\end{align*}
	In particular,  any positive $H$-superharmonic function is strictly
	positive.
	
	Furthermore, the above Harnack constant $C(H,W,f)$ can be chosen
	such that for any $f\leq \tilde{f}$ we have
	$$C(H,W,\tilde{f})\leq C(H,W,f).$$
\end{lem}
\begin{proof} We may assume that $u\gneqq 0$ on $W$.
	Let
	$$d(x):=\sum_{y\in X}b(x,y)+q(x) -f(x),\qquad x\in X.$$
	Take $y_{0}\in W$ with
	$u(y_{0})>0$.  Then, the inequality $\H u\ge fu$ gives for all
	$x\sim y_{0}$
	\begin{align*}
	d(x)u(x)\ge \sum_{y\in X}b(x,y) u(y)\ge b(x,y_{0})u(y_{0}).
	\end{align*}
	We conclude that $d$ and $u$ are strictly positive for all $x\sim
	y_0$. Since
	$W$ is connected, $d$ and $u$  are strictly positive  on $W$. Let
	a path $x_{0}\sim\ldots\sim x_{n}$ in $W$ from the maximum of $u$ at
	$x_{\max}=x_{0}$ to the minimum at $x_{\min}=x_{n}$ be given. We
	obtain
	\begin{align*}
	u(x_{\max})\le u(x_{\min})\prod_{j=0}^{n-1}\frac{d(x_{j})}{b(x_{j},x_{j+1})}\,.
	\end{align*}
	Taking the minimum of the product on the right side (over  all
	possible paths in $W$) yields the constant $C(H,W,f)$. The monotonicity in
	$f$ follows as $\tilde{d}\leq d$ whenever $f\le \tilde{f} $ for $\tilde{d}$ defined with
	$\tilde{f}$.
\end{proof}
A graph $b$ over $X$ is called {\em locally finite} if for all $x\in
X$
$$\#\{y\in X \mid b(x,y)>0\}<\infty.$$

The following {\em Harnack principle} is an immediate consequence of the local Harnack
inequality and Fatou's lemma.
\begin{lem}[Harnack principle] \label{l:convergence}
	Let the graph  be a connected  and let $f\in
	C(X)$.
	Let $(u_{n})$ be a sequence of positive
	functions in $F(X)$ such that $ \H u_{n} \ge f u_{n} $ in $X$, and suppose that there is $o\in X$ such that $C^{-1}\leq u_{n}(o)\leq
	C$ for some $C>0$. Then, there exists a subsequence $(u_{n_{k}})$
	that converges pointwise to a strictly positive function $u \in F(X)$ such that $ \H u\ge f u $.
	
	Furthermore, assume that one of the following properties hold true:
	\begin{itemize}
		\item The graph $b$ is locally finite.
		\item $(u_{n_{k}})$ is monotone increasing in $k$.
		\item There exists $g\in \F(X)$ such that  $u_{n_{k}}\leq
		g$ for all $k\in \mathbb{N}$ in $X$.
	\end{itemize}
	Then, $\H u_{n_{k}}\to \H u$.
\end{lem}
\begin{proof}By the Harnack inequality the set of positive
	functions $v$ such that $ \H v\ge fv $ in $X$ such that $C^{-1}\leq v(o) \leq C$ for some fixed
	$o\in X$ and a positive constant $C>0$ is compact with respect to the product topology,
	that is with respect to pointwise convergence. Hence, the sequence
	$(u_{n})$ has a convergent subsequence $(u_{n_{k}})$ that converges
	to a strictly positive limiting function $u$. We are left to check that $u$ satisfies   $Hu\ge fu $. Since the functions $u_{n}$ satisfy $ Hu\ge fu $,  we have for all $x\in X$
	\begin{align*}
	\sum_{y\in X}b(x,y)u_{n_{k}}(y)\leq u_{n_{k}}(x)\Big(\sum_{y\in
		X}b(x,y)+q(x)-f(x)\Big).
	\end{align*}
	By Fatou's lemma we have
	\begin{equation*}
	\sum_{y\in X}b(x,y)u(y)
	\leq  \liminf_{k\to\infty} \sum_{y\in X}\!b(x,y)u_{n_{k}}(y)\leq
	u(x)\Big(\sum_{y\in
		X}b(x,y)+q(x)-f(x)\Big).
	\end{equation*}
	Thus, $u\in \F(X)$ and $\H u\ge fu$.
	
	If the graph is locally finite, then all involved sums are over
	finitely many terms only.  Therefore, we can interchange the sum with
	the limit. On the other hand, if $u_{n_{k}}(y)$, $y\in X$, are
	monotone increasing in $k$ (resp., $u_{n_{k}}\leq g\in F$), then we
	can apply instead of Fatou, the monotone convergence theorem of
	Beppo-Levi (resp., the dominated convergence theorem), to get the
	convergence $\H u_{n_{k}}\to \H u$.
\end{proof}
\subsection{The ground state transform}\label{sec:GST}
The ground state transform uses a positive harmonic function to
turn Schr\"odinger operators into Laplace-type operators.

For $v\in C(X)$, define the operator $T_{v}:C(X)\to C(X)$ as
\begin{align*}
T_{v}f:=vf.
\end{align*}
For strictly positive $v$ the operator  $T_{v}$ is bijective and
$T_{v}^{-1}=T_{v^{-1}}$. Given a strictly positive $v$, we define
the operator ${\H}_{v}$ on ${\F}_{v}(X):=T_{v}^{-1}\F(X)$ by
\begin{align*}
{\H}_{v}:=T_{v}^{-1} {\H} T_{v}\,.
\end{align*}
For $v\in \F(X)$ one immediately checks that the constant function
belongs to ${\F}_{v}(X)$ since $1=T_{v}^{-1}v$. By a direct computation we obtain
\begin{align*}
H_{v}f(x)
&=\frac{1}{v(x)^{2}}\sum_{y\in X}b(x,y)v(y)v(x)(f(x)-f(y)) +
\frac{(Hv)(x)}{v(x)}f(x), \qquad  f\in F_{v}(X).
\end{align*}
If in addition,  $v$ is $H$-harmonic (resp., $H$-superharmonic) in $X$, then
$${\H}_{v}1={\H}v=0 $$
(resp., ${\H}_{v}1\ge0$) in $X$. In this case, the operator $\H_v$ is
called a {\em ground state transform} of the operator $\H$ with
respect to $v$.

We start with the following observation.

\begin{lem}\label{l:Lap2SO}
	Let $u$ be a   $H$-(super)harmonic function and let $v$ be a
	non-vanishing positive function. Then the function $u/v$ is a
	$H_v$-(super)\-harmonic function.
\end{lem}
\begin{proof}
	We calculate straightforwardly that $\H_{v}(u/v)=v^{-1}\H u\ge0$.
\end{proof}

For $0 \lneqq v\in \F(X)$, we define the bilinear form $h_{v}: C_{c}(X)\times
C_{c}(X)\to \mathbb{R}$ via
\begin{align*}
h_{v}(\ph,\psi):=\frac{1}{2}\sum_{x,y\in X}
b(x,y)v(x)v(y)(\ph(x)-\ph(y))(\psi(x)-\psi(y)).
\end{align*}
Using the Cauchy-Schwarz inequality one immediately sees that the
right hand side converges by the virtue of $v\in \F(X)$ and
$\sum_{y}b(x,y)<\infty$, $x\in X$, for all $\ph,\psi\in C_{c}(X)$.
Furthermore, we denote the induced quadratic form also
by $h_{v}$. Whenever $v$ is a $\H$-(super)harmonic positive function we call
$h_{v}$ a \emph{ground state transform} of $h$. This terminology is justified by
the following proposition. There the relation between ${\H}_{v}$, $h_{v}$ and $H$, $h$, for solutions $ v$ of a equation  $Hv=fv$ is discussed. Indeed,  a ground state of $ H $ solves such an equation in the case $ f=0 $. This is a well known
fact; see e.g. \cite{BG,FS08,FSW08,FLW,HK,TH} for proofs in (closely)
related contexts.

\begin{pro}[Ground state transform]\label{t:GST}
	Let $v\in \F(X)$ be strictly positive, $f\in C(X)$ such that $\H v=f
	v$. Then
	\begin{align*}
	h(\ph,\psi)
	=h_{v}\left({\frac{\ph}{v},\frac{\psi}{v}}\right)
	+\langle f\ph,\psi\rangle_{1} \qquad \ph,\psi\in C_{c}(X),
	\end{align*}
	$ H_{v} \ph\in\ell^{2}(X,v^{2})$ for $\ph\in C_{c}(X)$,  and
	\begin{align*}
	h_{v}(\ph,\psi)=\langle{{\H}_{v}\ph},{\psi}\rangle_{v^{2}}
	-\langle f\ph,\psi\rangle_{v^2} \qquad \ph,\psi\in C_{c}(X),
	\end{align*}
	where $\langle{\cdot},{\cdot}\rangle_{v^{2}}$ is the scalar product
	of $\ell^{2}(X,v^{2})$.
\end{pro}
\begin{proof}
	The first formula follows by a direct computation (cf. \cite[Theorem~10.1]{FLW}  and
	\cite[Proposition~3.2]{HK}). For the second formula,
	note that $ H_{v} \ph\in\ell^{2}(X,v^{2})$, $ \ph\in C_{c}(X) $, is equivalent to  $ H_{v} 1_{\{x\}}\in\ell^{2}(X,v^{2})$, $ x\in X $, which  is equivalent to the functions $  y\mapsto b(x,y)/v(x) $  being in  $ \ell^{2}(X,v^{2})$, $ x\in X $. However, this follows from $ \sum_{y\in X} b(x,y)<\infty$, $ x\in X $. Furthermore,
	observe
	that by the Green formula
\begin{equation*}
h_{v}(\ph,\psi)=h(T_{v}\ph,T_{v}\psi)
	- \langle f\ph,\psi\rangle_{v^2}
	=\langle{{\H}_{v}\ph},{\psi}\rangle_{v^{2}} -\langle
	f\ph,\psi\rangle_{v^2} .\qedhere
\end{equation*}
\end{proof}

\begin{rem} A particular advantage of the ground state transform $h_{v}$
	with respect to a strictly positive $\H$-harmonic function  $v$ is
	that it is \emph{Markovian} (or has the \emph{Markov
		property}), that is,
	\begin{align*}
	h_{v}(0\vee \ph\wedge 1)\leq h_{v}(\ph),\quad \ph\in C_{c}(X),
	\end{align*}
	while for $h$ one only has
	\begin{align*}
	h(|\ph|)\leq h(\ph),\quad \ph\in C_{c}(X).
	\end{align*}
	By decomposing a function into its positive and negative part one
	sees that the inequality for the modulus can be deduced from the
	Markov property.
\end{rem}

\subsection{Proof of the Allegretto-Piepenbrink theorem}

\begin{proof}
The inequalities ``$\ge$'' in (a) and (b) follow directly
from the ground state transform applied to $L+q-\lm m$.

Let us turn to the inequality ``$\leq$'' in  (b)
when the ``$\max$'' is replaced by a ``$\sup$''. So, fix a
vertex  $x\in X$. Notice that for any $\lm< \lm_{0}(H^{(m)})$ and $\delta_{x}:=1_{\{x\}}/m(x)$,  the
functions
$$u_{\lm}=(H^{(m)}-\lm)^{-1}\delta_{x}$$ are in $D(H^{(m)})$. Moreover, $u_{\lm}$ are positive by Corollary~\ref{c:posivitity_preserving}.
As
$H^{(m)}=\frac{1}{m}{\H}$ on $D(H^{(m)})\cap \F(X)$, and
$D(H^{(m)})\subseteq \F(X)$ by assumption, we infer by
Corollary~\ref{c:H}
$$({\H}-\lm
m)u_{\lm}=m(H^{(m)}-\lm) u_{\lm}=m(x)\delta_{x}= 1_{\{x\}}\ge0.$$
Furthermore, by the local Harnack inequality (Lemma~\ref{l:localHarnack}),
 the functions $u_{\lm}$ are strictly positive.

Next we show that in (b) the ``$\sup$'' is in fact a ``$\max$''.  To show
this let
$$\lm'\leq\lm<\lm_{0}(H^{(m)})\quad\mbox{and} \quad
g_{\lm}:=u_{\lm}/u_{\lm}(x).$$
For $y\in X$, let $W_{y} $ be a
connected and finite set such that $x,y\in W_{y}$. Then, by the
local Harnack inequality, Lemma~\ref{l:localHarnack}, and in particular by the monotonicity of the local Harnack constant with respect to $\gl$, we have
$g_{\lm}(y)\leq C_{y}$ for all $y\in X$,
 where $C_{y}=C(H,W_{y},\lm')$.
 Hence, there is a sequence
$\lm_{n}\to\lm_{0}=\lm_{0}(H^{(m)})$ such that $(g_{\lm_{n}})$
converges pointwise to a  limit $0\leq g\in C(X)$. By
$(H-\lm_{n}m)g_{\lm_{n}}\ge0$, we have for all $y\in X$
\begin{align*}
    \sum_{z\in X}b(y,z)g_{\lm_{n}}(z)
    \leq g_{\lm_{n}}(y)\left({\sum_{z\in X}b(y,z)}+q(y)-\lm_{n}m\right).
\end{align*}
The right hand side converges and, therefore, by Fatou's lemma $g\in
\F(X)$ and $(H-\lm_{0} m)g\ge0$.

Let us turn to the statement (c) concerning the essential spectrum.
Recall that by Proposition~\ref{p:H}, the assumption
$\frac{1}{m}{\L}C_{c}(X)\subseteq \ell^{2}(X,m)$ implies that
$D(H^{(m)})\subseteq \F(X)$.

For $W\subseteq X$, let $h_{W}$ be the form
acting as
$$h_{W}(\ph):=h(1_{ W} \ph),\qquad\ph\in C_{c}(X).$$
Furthermore, let
\begin{align*}
    \partial W:= W\times (X\setminus W).
\end{align*}
By
Lemma~\ref{l:P_K} below, the operator associated to the closure of the
form $(h-h_{X\setminus K})$ in $\ell^{2}(X,m)$ is compact whenever
$\frac{1}{m}{\L}C_{c}(X)\subseteq \ell^{2}(X,m)$ and $K\subseteq X$
is finite. Thus, the operator $H_{X\setminus K}^{(m)}$ associated
to $h_{X\setminus K}$  is a compact perturbation of $H^{(m)}$.
Hence,
$$\lm_{0}({H_{X\setminus K}^{(m)}})\leq \lm_{0}^{\mathrm{ess}}({H_{X\setminus K}^{(m)}})=\lm_{0}^{\mathrm{ess}}({H^{(m)}}).$$
Combining this fact
with a  Persson-type theorem, see e.g. \cite[Proposition~2.1]{HKW},
one finds that $$\sup_{\substack{K\subset X\mbox{ \scriptsize finite}}}\lm_{0}({H_{X\setminus
K}^{(m)}})=\lm_{0}^{\mathrm{ess}}({H^{(m)}}).$$ Now, the theorem
follows by applying (b) to $H^{(m)}_{(X\setminus
K)}$.
\end{proof}
\begin{lem}\label{l:P_K}
Let $b$ be a graph over $(X,m)$, let $q$  be a potential
and  let $K\subseteq X$ be a finite set. The form $(h-h_{X\setminus K})$ acts on
$C_{c}(X)$ as
\begin{align*}
    (h-h_{X\setminus K})(f)
    =h_{K}(f)-\sum_{(x,y)\in\partial K}b(x,y)f(x)f(y).
\end{align*}
If  $\frac{1}{m}LC_{c}(K)\subseteq \ell^{2}(X,m)$, then
$(h-h_{X\setminus K})$ is  bounded on  $\ell^{2}(X,m)$, and the
associated self-adjoint operator  is compact.
\end{lem}
\begin{proof}The formula for $    (h-h_{X\setminus K})$ follows by a
direct computation. For finite $K$, the form $h_{K}$ is zero outside
of the finite dimensional space $C_{c}(K)$, and therefore, bounded.
The assumption $\frac{1}{m}LC_{c}(K)\subseteq \ell^{2}(X,m)$ is
equivalent to the functions $\ph_{x}:X\to[0,\infty)$, $y\mapsto
b(x,y)/m(y)$ being in $\ell^{2}(X,m)$ for every $x\in K$.
So, for $f\in C_{c}(X)$ with
$\|f\|_m=1$ we have
\begin{align*}
    \left|\sum_{(x,y)\in\partial K,}b(x,y)f(x)f(y)\right|&\leq
\|f1_{K}\|_{\infty}
    \left|\sum_{(x,y)\in \partial K}{b(x,y)}f(y)\right|\\
    &\leq \|f1_{K}\|_{\infty}\sum_{x\in K}\langle{\ph_{x}},|f|\rangle_m\\
    &\leq \max_{x\in
    K}\frac{1}{m(x)}\max_{x\in K}\| \ph_{x}\|_m \#K.
\end{align*}
Hence, the form $(h-h_{X\setminus K})$ is bounded. We can associate
a self-adjoint operator $T_{K}$ to the closure of $(h-h_{X\setminus
K})$ on $\ell^{2}(X,m)$ and let $H_{K}^{(m)}$ be the finite
dimensional operator associated to the closure of $h_{K}$ on
$\ell^{2}(X,m)$. Set $P_{K}:=T_{K}-H_{K}^{(m)}$. Clearly, $T_{K}$ is
compact if  $P_{K}$ is compact. Observe that $P_{K}$ acts as
\begin{align*}
    P_{K}f(x)=\frac{1}{m(x)}\sum_{(x,y)\in\partial K}b(x,y)f(y), \qquad x\in X.
\end{align*}
Let $(f_{n})$ be a normalized  sequence that converges weakly in $\ell^{2}(X,m)$ to
$0$. Then, again using $\ph_{x}\in\ell^{2}(X,m)$ and the finiteness of
$K$, we get
\begin{align*}
  \|P_{K}f_{n}\|^{2}_m=
    \sum_{x\in K}\left(\frac{1}{m(x)}\sum_{(x,y)\in\partial
    K}b(x,y)f_n(y)\right)^{2}m(x)
\leq  \#K\max_{x\in K} \frac{\langle{\ph_{x}},{f_{n}}\rangle_m^{2}}{m(x)}\xrightarrow[n\to\infty]{} 0.
\end{align*}
Thus, $P_{K}$ is compact and so is $T_{K}$ which
finishes the proof.
\end{proof}
The following corollary is a ``measure-free'' version of the Allegretto-Piepenbrink theorem.
\begin{cor}[Allegretto-Piepenbrink theorem -- measure free version]
Let $b$ be a  connected graph over $X$ and let $q$  be a potential such that $ h\ge0 $. Then,
\begin{align*}
   \inf_{\ph\in C_{c}(X), \langle{\ph},{\ph}\rangle=1} h(\ph)    &=\max\{\lm\in \R\mid ({\H}-\lm )\ge0\}.
\end{align*}
\end{cor}
\begin{proof}
Let $m= 1$. We  close the form $h$ on $\ell^{2}(X)$ and consider the
associated positive self-adjoint operator $H^{(1)}$. By the
variational characterization of the bottom of the spectrum,  the
left hand side  is equal to $\lm_{0}(H^{(1)})$. By Proposition~\ref{p:H}~(a2)
we have $D(H^{(1)})\subseteq \F(X)$.  Thus, the asserted equality is an immediate consequence of  Theorem~\ref{t:AP}~(b).
\end{proof}
\section{Characterization of criticality}\label{s:char_crit}

In this section we discuss the notion of  criticality by giving
various characterizations of this notion. It turns out that being
critical coincides with the notion of recurrence in the case
$q=0$.

In this subsection we define the notions of criticality, subcriticality
and null-criticality that are fundamental for the present paper. In the
continuum context these notions go back to
B.~Simon who coined the terms {\em sub/super/critical-operators} for Schr\"odinger operators with short-range potentials which are
defined on $\R^d$, $d\geq 3$ \cite{Si0}.
These notions were  generalized by M.~Murata \cite{M86}
to Schr\"odinger operators which are defined in any subdomain of $\mathbb{R}^d$, where $d \geq 1$, and to general linear second-order elliptic operators with real coefficients by Y.~Pinchover
\cite{P88}. For more details see \cite{P07}. We remark that in the case of diffusion processes, transience
(resp., positive-recurrence, null-recurrence) are the analogous notions to subcriticality (resp., positive-criticality, null-criticality) in the context of Schr\"odinger
operators; we refer here to the monograph \cite{Pins95} (cf. Section~\ref{s:char_crit}).

Any function $ w:X\to \R $ gives rise to a canonical quadratic form on $ C_{c}(X) $ which we denote (with a slight abuse of notation) by $w$. It acts as
\begin{align*}
w(\ph):=\sum_{x\in X}w(x)\ph(x)^{2}.
\end{align*}

\begin{defin}[Critical/subcritical]Let $h$ be a quadratic form associated with a formal
	Schr\-\"odinger operator ${\H}$ such that $h\!\ge\!0$ on $C_{c}(X)$. The
	form $h$ is called \emph{subcritical} in $X$ if there is a positive
	$w\!\in\! C(X)$ such that $h-w\!\ge\!0$ on $C_{c}(X)$. Otherwise, the form
	$h$ is called \emph{critical}  in $X$.
\end{defin}

The following characterization of criticality is well known in
various contexts. For linear and quasilinear elliptic operators on $\R^{d}$, we refer to the following review papers \cite{P07,P13,PT,Pins95} and references therein. Within the theory of random walks and Dirichlet forms this
phenomenon appears under the name recurrence, see e.g.
\cite{Woe-Book,Fuk}. For positive matrices a corresponding phenomenon
is called $r$-recurrence, \cite{Pru64,VeJo67}. Furthermore, in the
context of Riemannian manifolds such a notion appears under the name
parabolicity. For the convenience of the reader we give a short
and self-contained proof for our setting. In Remark~\ref{r:4.5} below we discuss how the
result relates precisely to the ones in the other contexts that
include the discrete setting.

\subsection{Criticality}

We need the following definition.
\begin{defin}[Null-sequence] Let $ h\ge0 $ on $ C_{c}(X) $.
A sequence $(e_{n})$  in  $ C_{c}(X)$ of positive functions is called a {\em null-sequence} of $ h $ if there exists $o\in X$ and $c>0$ such that $e_{n}(o)=c$ for all $n\geq 1$ and
$h(e_{n})\to0$.
\end{defin}

\begin{thm}[Criticality characterization]\label{t:char_criticality}
Let $b$ be a  connected graph over $X$, let $q$ be a potential such that $h\ge0$ on $C_{c}(X)$. Then, the following are equivalent:
\begin{itemize}
  \item [(i)] $h$ is critical in $X$.
    \item [(ii)]
  $\lim_{\lm \nearrow\,  0}(H^{(1)}-\lm)^{-1}1_{\{x\}}(y)=\infty$
  for some (all) $x,y\in X$.
  \item [(ii')] $\lim_{t \to \infty}\int_{0}^{t}\mathrm{e}^{-tH^{(1)}}1_{\{x\}}(y)\,\mathrm{d}t
  =\infty$ for some (all) $x,y\in X$.
  \item [(iii)] There is a unique positive
  $\H$-superharmonic function  in $X$ (up to linear dependence).
\item [(iii')]
There are only finitely many linearly independent positive $\H$-super\-harmonic functions  in $X$.
\item [(iv)]
 For any $o\in X$ and $c>0$, there is a null-sequence $(e_{n})$ such that $e_{n}(o)=c$ for all $n\geq 0$.
\item [(iv')] There  exists a positive $\H$-harmonic function $ v $ in $X$ and a null-sequence $(e_{n})$
satisfying $0\leq e_{n}\leq v$ and $e_{n}(x)\to v(x)$  for all $x\in X$.
\item [(v)] $\mathrm{cap}_{h}(\{x\}):=\inf \{h(\ph)\mid \ph\in C_{c}(X),
\ph(x)=1\}=0 $ for all $x\in X$.

\end{itemize}
In particular, if one of the equivalent conditions holds true, then
the unique (up to linear dependence) positive $\H$-superharmonic
function on $X$ is $\H$-harmonic.
\end{thm}
\begin{defin}[Ground state]
Let a critical form $ h $ be given. For a given $ o\in X  $ we call the  positive $\H$-harmonic function $ \psi$ with $ \psi(o)=1 $ given by the theorem above the (Agmon) \emph{ground state} of $ h $ normalized at $ o $.
\end{defin}

We extract some lemmas from the proof that we will use also in later parts of this paper.

\begin{lem}\label{l:harmonic_crit} Let  $h\ge0$ on $C_{c}(X)$ be critical. Then there exists a non-trivial positive $ \H $-harmonic function and every positive $ \H $-superharmonic function is $\H$-harmonic.
\end{lem}
\begin{proof} 	
	By part (a2) of Proposition~\ref{p:H},
	$D(H^{(1)})\subseteq F(X)$. Hence,  the Allegretto-Piepenbrink theorem
	implies that there exists a positive $\H$-super\-harmonic function
	$v$. 	
	By the local  Harnack inequality, Lemma~\ref{l:localHarnack}, $ v $ is strictly positive.
	Letting $w:=({\H v})/{v}$, we
	observe that
	\begin{align*}
	(H-w)v=0.
	\end{align*}
	Hence, by the ground state transform, Proposition~\ref{t:GST}, we observe
	\begin{align*}
	(h-w)(\ph)= (h-w)_{v}(\ph/v)\ge0,\qquad \ph\in C_{c}(X),
	\end{align*}
	Hence, by criticality $(\H v)/v=w\equiv0$.
	Thus, $v$ is a positive $\H$-harmonic function.
\end{proof}

\begin{lem}\label{l:GFbound} Let $b$ be a connected graph over $X$, let $q$ be a potential and $ w\ge0 $ be such that $h-w\ge0$ on $C_{c}(X)$ and let $ m $ be a measure of full support on $ X $. Then 	\begin{align*}
w(x)(H^{(m)}-\lm)^{-1}1_{\{x\}}(x) \leq 1, \qquad  \lm<0 \mbox{ and } x\in X.
	\end{align*}
\end{lem}
\begin{proof} 	We let $h^{(m)}$ be the closure of $h$ in $\ell^{2}(X,m)$. Then, for 	all $f\in D(h^{(m)})$ with approximating sequence $ (f_{n}) $ in $ C_{c}(X) $ we have by Fatou's lemma
	$$ h^{(m)}(f)=\lim_{n\to\infty}h^{(m)}(f_{n})\ge
	\liminf_{n\to\infty}\langle
	wf_{n},f_{n}\rangle_{m}\ge \langle
	wf,f\rangle_{m}. $$
	Let $H^{(m)}$ be the positive self-adjoint
	operator associated to $h^{(m)}$. For $\lambda<0$ and a fixed $x\in X$, denote $g_{\lm}:=(H^{(m)}-\lm )^{-1}1_{\{x\}}$.
	We find that for all $\lm<0$
\begin{align*}
m(x)g_{\lm}(x)&=  \langle 1_{\{x\}}, g_{\lm}\rangle_{m}
=  \langle (H^{(m)}-\lm) g_{\lm}, g_{\lm}\rangle_{m}
 \\&
 \geq h^{(m)}( g_{\lm}) \ge \langle w g_{\lm}, g_{\lm} \rangle_{m}
	\ge m(x)w(x)g_{\lm}(x)^{2}.
\end{align*}
Since $ g_{\lm}(x)>0 $ (by Harnack's inequality), and as $m$ has full support, the statement follows.
\end{proof}

\begin{proof}[Proof of Theorem~\ref{t:char_criticality}]
	Let us first comment on the limit in (ii) and the equivalence of the formulation with ``some'' and ``all'': By the resolvent formula and positivity preservation of the resolvents (Corollary~\ref{c:posivitity_preserving}), the map $ \lm\mapsto (H^{(1)}-\lm)^{-1}1_{x}(y) $ is monotone increasing in $ \lm $: precisely, for $\lm \geq \mu$, one has
	$$ (H^{(1)}-\lm)^{-1}1_{x} - (H^{(1)}-\mu)^{-1}1_{x} = (\lm - \mu)(H^{(1)}-\lm)^{-1}(H^{(1)}-\mu)^{-1} 1_{x} \geq 0.$$
		 Hence, the limit either exists and is positive or it is infinite. Moreover, the functions $ g_{\lm}=(H^{(1)}-\lm)^{-1}1_{x} $ are $ \H $-superharmonic (since $ H^{(1)} $ is a restriction of $ \H $ by Corollary~\ref{c:H}). Assuming the limit exists at one vertex $ y $, these functions satisfy the assumptions of the Harnack principle, Lemma~\ref{l:convergence}. Hence, the limits exist for all vertices, as $ g_{\lm} $ then converges to a $ \H $-superharmonic function. By the symmetry of $ x,y $ in $ (H^{(1)}-\lm)^{-1}1_{x}(y)  $ due to self-adjointness of $ H^{(1)} $ on $ \ell^{2}(X) $ the limit then also exists for all $ x $.
	\medskip
	
	Let us now turn to show the equivalences.
	\medskip
	
	(ii) $\Longleftrightarrow$ (ii'): This equivalence follows from Lemma~\ref{l:Laplace} (the
	Laplace transform).
	\medskip
	
	(iv) $\Longleftrightarrow$ (v): This simply follows from the definition
	of $\mathrm{cap}_{h}$.
	\medskip

	We proceed by proving that  (i) $\Longleftrightarrow$ (iv).
	\medskip
	
	(i)$\Longrightarrow$(iv): Let $w_{n}=\frac{1}{n} 1_{\{o\}}$ for $o\in
	X$. Then the  criticality implies for all $ n\in\N $ the existence of $0\leq e_{n}\in
	C_{c}(X)$ such that
	\begin{align*}
	0\le    h(e_{n})< \sum_{x\in X}w_{n}(x)e_{n}(x)^{2}= \frac{1}{n}
	e_{n}(o)^{2}.
	\end{align*}
	In particular, $e_{n}(o)>0$ and, thus, $(e_{n})$ can be chosen such
	that $e_{n}(o)=c$. Hence, $h(e_{n})\leq \frac{c^2}{n}$ which implies
	that $(e_{n})$ is a null-sequence.
	
	\medskip

(iv) $\Longrightarrow$ (i): Let $w\ge 0$ such that $h\ge w$ on $C_{c}(X)$.
Choose $o\in X$ and $(e_{n})$ as assumed. Then,
$$0=\lim_{n\to\infty}h(e_{n})\ge
\limsup_{n\to\infty}\sum_{x\in
	X}w(x)e_{n}(x)^{2}\ge\limsup_{n\to\infty}
w(o)e_{n}(o)^{2}=w(o)c^{2}.$$ Since $c>0$, we infer $w(o)=0$. As $o$
can be chosen arbitrarily, we conclude  $w\equiv0$.
		\medskip	
	
We show	next that (i)\,\&\,(iv) $\Longrightarrow$ (iv') $\Longrightarrow$ (iii)
	$\Longrightarrow$ (iii')  $\Longrightarrow$ (ii)  $\Longrightarrow$ (i).
		
\medskip

(i)\,\&\,(iv) $\Longrightarrow$ (iv'):
 Let $v$ be a positive $\H$-harmonic function whose existence is guaranteed by Lemma~\ref{l:harmonic_crit}. Let
$(e_{n})$ be a null-sequence such that $e_n(o)=v(o)$, where $o\in X$. Let $h_{v}$ be
the ground state transform with respect $v$. Then, since $h_{v}$ is Markovian
\begin{align*}
h(e_{n}\wedge v)=  h_{v}( v^{-1}(e_{n}\wedge v))=    h_{v}(
v^{-1}e_{n}\wedge 1)\leq h_{v}(
v^{-1}e_{n})=h(e_{n})\xrightarrow[n\to\infty]{} 0.
\end{align*}
Hence, we can replace $e_{n}$ by $\ow
e_{n}:=e_{n}\wedge v$. Furthermore, $h_{v}(v^{-1}\ow e_{n})\to 0$, thus, $ (v^{-1}\ow e_{n}) $ converges pointwise to a constant and since
$e_{n}(o)=v(o)$ have $v^{-1}\ow e_{n}\to 1$ pointwise as $n\to\infty$.	

\medskip

(iv') $\Longrightarrow$ (iii): Let $ v $ be the positive $ \H $-harmonic function and $ (e_{n}) $ be the null sequence such that $ e_{n} \to v$ pointwise as $ n\to\infty $ whose existence is assumed by (iv'). Let $ u $ be a non-trivial positive $ \H $-superharmonic function which is strictly positive by the local Harnack inequality, Lemma~\ref{l:localHarnack}. Then, by the ground state transform, we infer,
\begin{align*}
0\leq h_{u}\left(\frac{e_{n}}{u}\right)\leq h(e_{n})\xrightarrow[n\to\infty]{} 0.
\end{align*}
Hence, there is a constant $c>0  $ such that for all $ x\in X $
$$ c=\lim_{n\to\infty} \frac{e_{n}(x)}{u(x)}  =\frac{v(x)}{u(x)}\,.$$
Thus, $ v $ and $ u $ are linearly dependent.

\medskip

(iii) $\Longrightarrow$ (iii'):  This is trivial.

\medskip

(iii') $\Longrightarrow$ (ii): Suppose that the limit in (ii) is finite for
all $x,y\in X$. Then for $x\in X$ the limit $G(x,\cdot)$ is a
positive $H$-superharmonic function such that  $\H G(x, \cdot)=1_{\{x\}}$ by
Lemma~\ref{l:convergence},. Hence,
via $G(x,\cdot)$, $x\in X$, there are infinitely many linearly independent
positive $H$-super\-harmonic functions, and this contradicts (iii').

\medskip

(ii) $\Longrightarrow$ (i): Assume that $h$ is subcritical. Then, there
exists a positive $w$ such that $h-w\ge0$ on $C_{c}(X)$.  Then, for all $\lambda<0$ and a fixed $x\in
\mathrm{supp}\, w$, we have by Lemma~\ref{l:GFbound}
\begin{equation*}
\lim_{\lm\nearrow 0}(H^{(1)}-\lm )^{-1}1_{\{x\}}(x)\le \frac{1}{w(x)}<\infty. \qedhere
\end{equation*}
\end{proof}
\begin{rem}
Statement (iv') of Theorem~\ref{t:char_criticality} seems to be not known in the context of symmetric second-order elliptic differential operators, but indeed, it follows by the same argument as in the proof above.
\end{rem}

There is a closely related phenomena called recurrence which is vastly studied in the literature. In the remark below we discuss how the results concerning recurrence
for Dirichlet forms, random walks and positive matrices are related
to our setting.
\begin{rem}\label{r:4.5}
(a) First we look at random walks and positive matrices. Note that
the operator $L$ in the decomposition of $H=L+q$ can be further
decomposed into $L=D-A$ such that $D$ is the multiplication operator  by
the weighted vertex degree $d(x)=\sum_{y}b(x,y)$, $x\in X$, and $A$ is the weighted adjacency matrix with entries $b(x,y)$, $x,y\in
X$. In random walks, one studies the transition matrix of a graph
that is given by $P=D^{-1}A$. Then, recurrence is defined via the
divergence of the sum $\sum_{n}P^{n}$ which can shown to be
equivalent  in the case $q=0$ to condition (ii) of Theorem~\ref{t:char_criticality}. As we allow
for non-vanishing $q$ our setting is more general. However, the
setting
of positive matrices, see e.g. \cite{Pru64,VeJo67}, includes
the one of random walks and is much
more flexible. There irreducible
matrices with positive entries are studied. Given our setting of a
connected graph, we observe that $h(1_{\{x\}})=\deg(x)+q(x)$, $x\in
X$. Hence, $\deg+q\ge 0$ is implied by  $h\ge0$ and even $\deg+q> 0$
since otherwise $1_{\{x\}}$ would be an eigenfunction and, therefore, $x$
is an isolated vertex. Hence, we can consider the matrix
$(D+q)^{-1}A$, and this brings us into the realm of the theory of
positive matrices.
In this context the equivalence of (iv) and (v)
and the existence of a unique positive solution have been proven in \cite{Pru64,VeJo67}.

(b) For Dirichlet forms the equivalences of the theorem are well
known provided there is a suitable notion of superharmonic
functions, see e.g.  \cite{Fuk}  except for (iii) and (iii'). This framework includes forms $h$ on graphs $b$ with
potentials $q\ge0$, see \cite{KL1}. As we allow for non-positive
$q$, our theory is a priori not included in the latter work.
However, once one has
shown the existence of a positive harmonic function the results can
be carried over using the ground state transform. Unfortunately, this can be  guaranteed a priori only  for locally finite graphs, see \cite{HK}.
\end{rem}

\subsection{The extended space}
In this section we extend the form $ h $ to a  larger set of functions by taking an appropriate  closure of $ C_{c}(X) $. In the critical case we obtain an explicit representation of the action of $ h $ on this closure. Furthermore, this allows us formulate an elegant negation of statement (iv) of the Theorem~\ref{t:char_criticality} to characterize subcriticality below.

Let $h\geq 0$ on $C_{c}(X)$ for a connected graph $ b $ and a potential $ q $ and fix $o\in X$. Define
 \begin{align*}
 \|\ph\|_{h,o}:=\sqrt{h(\ph)+\ph^2(o)}, \qquad \ph\in C_{c}(X)
 \end{align*}
 which is a priori  a seminorm on $C_{c}(X)$. By the Green formula
 and the Harnack inequality we even have positive definiteness on
 $C_{c}(X)$, and we obtain a norm.
 \begin{lem}\label{lem_pointwise} Let $b$ be a connected graph over $X$, let
 	$q$ be a potential such that $h\ge0$ on $C_{c}(X)$.
 	\begin{itemize}
 		\item [(a)] The norms $\|\cdot\|_{h,o}$
 		and $\|\cdot\|_{h,o'}$ are equivalent, for any $o$ and $o'$ in $X$.
 		\item [(b)] Convergence with respect to $\|\cdot\|_{h,o}$ implies
 		pointwise convergence.
 	\end{itemize}
 \end{lem}
 \begin{proof}
 	Fix a positive $H$-supersolution $u$ on $X$. Let
 	$o,o'\in X$, and let a path $o=x_{0}\sim\ldots\sim x_{n}=o' $ be
 	given. Set
 	\begin{align*}
 	B_{o,o'}:=\left(\sum_{i=0}^{n-1} \frac{1}{b(x_{i},x_{i+1})u(x_{i})u(x_{i+1})}\right).
 	\end{align*}
 	Then, using a telescoping sum argument along the path and the
 	Cauchy-Schwarz inequality yields by means of the ground state transform (see, Proposition~\ref{t:GST}),
 	\begin{align*}
 	\left|\frac{\ph(o)}{u(o)}-\frac{\ph(o')}{u(o')}\right|^{2}&\leq
 	B_{o,o'}h_{u}\left(\frac{\ph}{u}\right)\leq B_{o,o'}h(\ph),
 	\end{align*}
 	for all $\ph\in C_{c}(X)$. Thus,
 	$\left|\frac{\ph(o)}{u(o)}-\frac{\ph(o')}{u(o')}\right|^2$ can be
 	controlled by $h(\ph)$, and therefore, the norms $\|\cdot\|_{h,o}$
 	and $\|\cdot\|_{h,o'}$  are equivalent which is (a). Statement (b)
 	follows directly from (a).
 \end{proof}

 We denote by $D_{h,o}$ the completion of $C_{c}(X)$ with respect to
 $\|\cdot\|_{h,o}$. Note that in general such a space might not be a function space. But Lemma~\ref{lem_pointwise} implies that $D_{h,o}$ is a function space and hence a subspace of $ C(X) $. So,
 $$D_{h,o}:=\ov{C_{c}(X) }^{\|\cdot\|_{h,o}}\subset C(X). $$
 In analogy to the theory of Dirichlet forms, we call it the \emph{extended space } of the form $ h $.

 Note that by Lemma~\ref{lem_pointwise} above,
 $D_{h,o}$ does not depend on the choice of $o$ due to connectedness. Clearly, $h$ is
 closable on $D_{h,o}$ and we denote the closure by  $ \overline
 {h}$, i.e., for a sequence $ (\ph_{n}) $ in $ C_{c}(X) $ that converges with respect to $ \|\cdot\|_{{h,o}} $ to a function $ \ph\in D_{h,o} $ we let
 \begin{align*}
 \ov h(\ph):=\lim_{n\to\infty}h(\ph_{n}).
 \end{align*}
By polarization  $ \ov h $ induces a bilinear form on $\ov h :D_{h,o}\times D_{h,o}\to \R $ and, hence,
$(D_{h,o},\ov h )$ becomes a semiscalar product space.

In the critical case there is an alternative way to define an extension of a form $ h $ that satisfies $ h\ge0 $ on $ C_{c}(X) $. Let $ v $ be the ground state of the critical form $ h $ and let $ h_{v} $ be its ground state transform.
Note that since $ h_{v} $ consists of a sum involving positive terms only,  we can define $ \ow {h}_{v}:C(X)\to [0,\infty] $ via
\begin{align*}
\ow h_{v}(f):=\frac{1}{2}\sum_{x,y\in X}b(x,y)v(x)v(y )(f(x)-f(y))^{2}
\end{align*}
and
\begin{align*}
\ow D_{h_{v}}=\{f\in C(X)\mid \ow h_{v}(f)<\infty\}.
\end{align*}
By Fatou's lemma $ \ow h_{v}  $ is lower semi-continuous and, thus, closed on $ D_{h_{v}} $.
This observation allows us to define $ \ow {h}:C(X)\to [0,\infty] $ via
\begin{align*}
\ow h(f):= \ow h_{v}\left(\frac{f}{v}\right)=\frac{1}{2}\sum_{x,y\in X}b(x,y)v(x)v(y )\left(\frac{f(x)}{v(x)}-\frac{f(y)}{v(y)}\right)^{2},
\end{align*}
and
\begin{align*}
\ow D_{h}:=\{f\in C(X)\mid \ow h(f)<\infty\}=v^{-1}\ow D_{h_{v}}.
\end{align*}
The next theorem shows that in the critical case the forms $ \ov h $ and $ \ow h $ coincide.
\begin{thm}\label{t:formcrit}
 Let $b$ be a connected graph over $X$, let
 $q$ be a potential such that $h\ge0$ on $C_{c}(X)$ is critical and $ o\in X $. Then
 \begin{align*}
D_{h,o} =\ow D_{h},
 \end{align*}
 and for $ f\in D_{h,o} $
 \begin{align*}
 \ov h(f)=\ow h(f).
 \end{align*}
\end{thm}

To prove this theorem we use some standard arguments from Dirichlet form theory. For the convenience of the reader we give the proofs.

\begin{lem}\label{l:approximation_by_nullseq}
	Let $q=0$, and $b$ be a graph such that $h\ge0$ is critical. Let
	$(e_{n})$ be a null-sequence for $h$, satisfying  $0\le e_{n}\le 1$. Then for any  function $f$ we
	have
	\begin{align*}
	\lim_{n\to\infty}h(e_{n}f)= \ow h(f).
	\end{align*}
\end{lem}
\begin{proof}
	By Fatou's lemma we have
	\begin{align*}
	h(f)\leq\liminf_{n\to \infty}    h(fe_{n}).
	\end{align*}
	Thus, in order to show the reverse inequality it suffices to assume
	that $h(f)<\infty$.
	Assume first that $ f $ is bounded.
	Using the elementary inequality
	$$(ab-cd)^{2}=[a(b-d)+d(a-c)]^{2}\leq
	2a^{2}(b-d)^{2}+2d^{2}(a-c)^{2},$$ for all $  a,b,c,d\in \R $ and Tonelli's theorem to
	obtain
	\begin{align*}
	h((1-e_{n})f)=&\frac{1}{2}\sum_{x,y\in
		X}b(x,y)(((1-e_{n})f)(x)-((1-e_{n})f)(y))^{2}
	\\
	\leq&\sum_{x,y\in X}b(x,y)(1-e_{n})^{2}(x)(f(x)-f(y))^{2}
	\\
	&+\sum_{x,y\in X}b(x,y)f(y)^{2}((1-e_{n})(x)-(1-e_{n})(y))^{2}\\
	\leq& \sum_{x,y\in
		X}b(x,y)(1-e_{n})^{2}(x)(f(x)-f(y))^{2}+\|f\|_{\infty}^{2}
	h(e_{n}).
	\end{align*}
	
	Letting $n\to\infty$, the right hand side tends to $0$. Indeed, the
	second term tends to $0$, since  $0\le e_{n}\le 1$ is a
	null-sequence, while the first term tends to $0$ by the dominated
	convergence theorem since $h(f)<\infty$ (by the above assumption), and $e_{n}\to 1$
	as $n\to\infty$, pointwise (use $F(x,y):=(f(x)-f(y))^{2}$ as a
	dominating function). On the other hand, by the reverse triangle
	inequality, we have
	$$|
	h(f e_{n})^{{1}/{2}}- h(f)^{{1}/{2}}|\leq h(f(1-e_{n}))^{{1}/{2}},$$
	and the statement follows for bounded $ f $.
	
	To obtain the statement for arbitrary $ f $, note that
	\begin{align*}
	\lim_{k\to\infty}\ow h(-k\vee f\wedge k)=\ow h(f).
	\end{align*}
	Indeed, the inequality ``$\ge$'' follows by Fatou's lemma.  The inequality
	``$\leq$'' follows since we have $\ow h_{v}(f_{k})\leq \ow h_{v}(f)$, where $f_k:=-k\vee f \wedge k$ (which is
	the Markov property of $ h(-k\vee f \wedge k)\leq h(f)$). 	This finishes the proof.
\end{proof}

\begin{proof}[Proof of Theorem~\ref{t:formcrit}]
Let $ v $ be the ground state of $ h $ and $ h_{v} $ the ground state transform.

Let	$ f\in  D_{h_{v},o} $ and let $ (f_{n}) $ be an approximating sequence in $ C_{c}(X) $ with respect to $ \|\cdot\|_{h_{v},o} $. By the Lemma~\ref{lem_pointwise} we have pointwise convergence $ f_{n}\to f $.
Thus, by Fatou's lemma we have
\begin{align*}
\ow h_{v}(f)\leq \lim_{n\to\infty }h_{v}(f_{n})=\ov h_{v}(f)
\end{align*}
which in particular yields
$ D_{h_{v},o}\subseteq  \ow  D_{h_{v}}$.

	Since $ h_{v} $ is critical (as $ h $ is critical),  let $ (e_{n}) $ be a null-sequence converging to $ 1$ pointwise (which exists according to Theorem~\ref{t:char_criticality}~(iv')).
	
	Let $ f\in    \ow  D_{h_{v}}$.
	Then, since $\ow h_v(f) < \infty$, we can show $\ow h_v(f-e_n f) \to 0$ in the same way as in the proof of the previous Lemma~\ref{l:approximation_by_nullseq}:
	one starts with bounded $f$ and shows convergence to zero via dominated convergence; for general $f$, one uses approximations $f_k := -k \vee f \wedge k$.

	This yields $ f\in  D_{h_{v},o} $ and by the lemma above and Fatou's lemma
	\begin{align*}
	\ov h_{v}(f)\leq \lim_{n\to\infty }h_{v}(e_{n}f)=\ow h_{v}(f).
	\end{align*}
	By what we have proven we deduce via the ground state transform
\begin{align*}
 D_{h,o} =v^{-1} D_{h_{v},o}=v^{-1} \ow D_{h_{v}}=\ow D_{h},
\end{align*}	
and for $ f\in D_{h,o}$
\begin{equation*}
\ow h(f)=\ow h_{v}\left(\frac{f}{v}\right)=\ov h_{v}\left(\frac{f}{v}\right)=\ov h(f). \qedhere
\end{equation*}
\end{proof}

\subsection{Subcriticality}
In this section we reformulate Theorem~\ref{t:char_criticality} to give a characterization of subcriticality.

\begin{thm}[Subcriticality characterization]\label{t:subcriticality}
Let $b$ be a  connected graph over $X$, let $q$ be a
potential such that $h\ge0$ on $C_{c}(X)$. Then the following assertions are
equivalent:
\begin{itemize}
  \item [(i)] $h$ is subcritical in $X$.
  \item [(ii)]
  $0<\lim_{\lm \nearrow\,  0}(H^{(1)}-\lm)^{-1}1_{\{x\}}(y)<\infty$
  for some (all) $x,y\in X$.
  \item [(ii')] $0<\int_{0}^{\infty}\mathrm{e}^{-tH^{(1)}}1_{\{x\}}(y)\,\mathrm{d}t
  <\infty$ for some (all) $x,y\in X$.
  \item [(iii)] For every vertex $x\in X$ there exists a positive $H$-super\-harmonic function $u$ such that $\H u=1_{\{x\}}$ in $X$.
  \item [(iii')] There are infinitely many linearly independent positive $H$-super\-har\-monic functions  in $X$.
  \item[(iv)] The norms $\ov h(\cdot)^{1/2}$ and $\|\cdot\|_{h,o}$ are equivalent on $D_{h,o}$. In particular,  $(D_{h,o},\ov h)$ is a Hilbert space.
  \item[(iv')]  There is no positive $ \H $-harmonic function in $ D_{h,o}$.
\item [(v)] $\mathrm{cap}_{h}(\{x\})>0 $ for all $x\in X$.
\end{itemize}
\end{thm}

\begin{proof} The equivalences
(i) $\Longleftrightarrow\!$ (ii) $\Longleftrightarrow\!$ (ii') $\Longleftrightarrow$ (iii') $ \Longleftrightarrow$ (iv) $ \Longleftrightarrow$ (v) follow immediately from
Theorem~\ref{t:char_criticality}, where positivity in (ii) and (ii') is guaranteed by the Harnack principle (Lemma~\ref{l:convergence}), and the Laplace transform (Lemma~\ref{l:Laplace}).

The implication  (ii) $\Longrightarrow$ (iii) follows by
  Lemma~\ref{l:convergence} with a normalization at $x$. The implication
(iii) $\Longrightarrow$ (iii') is clear.

(iv)  $\Longrightarrow$ (iv'): Let $ v $  be a non-negative $ \H $-harmonic function  in $ D_{h,o} $ and let $ (v_{n})  $ be a sequence in $ C_{c}(X) $ converging to $ v $ with respect to $\|\cdot\|_{h,o}$. By Lemma~\ref{l:wlogmonotone}, we can assume $ |v_{n}|\leq v $, and therefore, by the virtue of Lebesgue's theorem $ \lim_{n\to\infty} \H v_{n}=\H v=0 $. So, for all $ \ph\in C_{c}(X) $, we conclude by the Green formula
\begin{align*}
\ov h(v,\ph)=\lim_{n\to\infty} h(v_{n},\ph)=\lim_{n\to\infty} \sum_{X}\ph\H v_{n}=0.
\end{align*}
Since $ \ov h $ is a scalar product by (iv) and $ C_{c}(X) $ is dense in $ D_{h,o} $, we infer $ v=0 $.

(iv') $\Longrightarrow$ (i): Assume $ h $ is critical. Then, by Theorem~\ref{t:char_criticality}~(iv') there is a positive harmonic function $ v $ and a null sequence $ (v_{n}) $ converging pointwise to $ v $. Hence, $ (v_{n}) $ is a $ \|\cdot\|_{h,o} $ Cauchy-sequence in $ D_{h,o} $. Lemma~\ref{lem_pointwise}~(b) and the pointwise convergence to $ v $ implies that $ v\in D_{h,o} $.
\end{proof}


\subsection{The Green function}\label{s:GF}
In this section we take a closer look at the Green function which already occurred in Theorem~\ref{t:subcriticality}~(iii). We show that it can be obtained also via an approximation of the Green function along an exhaustion of $ X $. Fixing one argument, we show that the Green function is the minimal positive solution of the equation $ \H g=1_{x} $, and that it belongs to $ D_{h,o} $.

\begin{defin}[\emph{Green function}]  Let $b$ be a  connected graph over $X$, let $q$ be a
	potential such that $h\ge0$ on $C_{c}(X)$. If the integral
	\begin{align*}
	G(x,y)=\int_{0}^{\infty}\mathrm{e}^{-tH^{(1)}}1_{\{x\}}(y)\,\mathrm{d}t
	\end{align*}
	converges at a point $(x,y)\in X\times X$, $x\neq y$, (or equivalently, at all points $(x,y)$, $x\neq y$), then $G$ is called the {\em positive minimal Green function} of $ h $ in $X$.
\end{defin}
Recall that by Theorem~\ref{t:subcriticality}~(iii), the positive minimal Green function of $h$ exists in $X$ if and only if $h$ is subcritical in $X$. The minimality of $G$ is demonstrated in Theorem~\ref{t:GF} (c).

Our aim is to show that $ G $ can be obtained as the limit along an exhaustion. To this end we have to restrict $ h $ and $ \H $ to finite sets and show that these restrictions are invertible. So, let $ K\subseteq  X $ be a  finite subset and let $ \H_{K} $ be operator associated to the form $ h $ restricted to $ C_{c}(K) $. Note that the operators $ \H_{K} $ are restrictions of $ \H $.

We need the following minimum principle.

\begin{lem}[Minimum principle]\label{l:minimum}	 Let $b$ be a  connected graph over $X$, let $q$ be a
	potential such that $h\ge0$ on $C_{c}(X)$. Let $ K\subseteq X $ be a finite set and $ u $ be $ \H $-superharmonic on $ K $ and nonnegative outside $K$. Then, $ u $ is nonnegative. Moreover, $u$ is either strictly positive in $K$ or $u=0$ in $K$.	
\end{lem}
\begin{proof} Let $ v $ be a strictly positive $ \H $-superharmonic function on $ K $. Then, by Lemma~\ref{l:Lap2SO}, the function $ u/v $ is $ \H_{v} $ superharmonic on $ K $. For $ \H_{v} $ we can apply the minimum principle for operators on graphs with positive potentials, see e.g. the proof of \cite[Theorem 8]{KL}, and find that $ u/v $ is nonnegative on $ K $. Since $ u $
was assumed to be nonnegative outside of $ K $ as well, $ u $ must be nonnegative, and by Harnack inequality $u$ is either strictly positive in $K$ or $u=0$ in $K$.
\end{proof}
\begin{lem}\label{l:4.11}
	  Let $b$ be a  connected graph over $X$, let $q$ be a
	  potential such that $h\ge0$ on $C_{c}(X)$ and let  $ K\subseteq  X $ be a  finite subset. Then, $ H_{K} $ is invertible on $ C_{c}
	  (K)=\ell^{2}(K) $.
\end{lem}
\begin{proof} Let  $ v $ be a strictly positive $ \H $-superharmonic function given by Theorem~\ref{t:AP} (Allegretto Piepenbrink-type theorem).
	Since $ K $ is finite, $ H_{K} $ not being invertible means there is $0 \neq u \in C_{c}(K)$ such that $ H_{K}u=0 $.
	Assume $ u\neq 0 $ and $u$ is strictly positive at one point in $K$. Then $ u $ is strictly positive on $ K $ by the minimum principle, Lemma~\ref{l:minimum}.
	Hence, by the ground state transform we have
	\begin{align*}
	0=\langle u ,H_{K}u\rangle=h(u)\ge h_{v}\left(\frac{u}{v}\right) \ge\sum_{x\in K,y\in X\setminus K}b(x,y)v(x)v(y)\frac{u(x)^{2}}{v(x)^{2}}\,.
	\end{align*}
Since $ v $ is strictly positive on $  X$ and $ u $ on $ K $, we infer that $ b(x,y)=0 $ for all $ x\in K $ and $ y\in X\setminus K $. This is a contradiction to the connectedness of the graph.
	\end{proof}
With this preparation we can formulate the following statements about the Green function. Here, we call a sequence of finite subsets $ (K_{n})$  of  $ X $ an \emph{exhaustion} of $ X $ if $ K_{n}\subseteq K_{n+1} $, $ n\in\N $ and  $ \bigcup_{n}K_{n}=X $.
\begin{thm}[Properties of the Green function]\label{t:GF}
	 Let $b$ be a  connected graph over $X$, let $q$ be a
	 potential such that $h\ge0$ on $C_{c}(X)$ and suppose that $ h $ is subcritical in $X$.
	 \begin{itemize}
	 	\item [(a)]We have  $$  	 G(x,y)=\lim_{\lm \nearrow\,  0}(H^{(1)}-\lm)^{-1}1_{\{x\}}(y)  $$ for all $ x,y\in X $ and the convergence is monotone increasing. In particular, $ G(x,y)=G(y,x) $ for all $ x,y\in X $,
 and  	$ G(x,\cdot) $ and $ G(\cdot,x) $ are superharmonic such that \begin{align*}
 \H G(x,\cdot) =\H G(\cdot,x) =1_{\{x\}},\qquad x\in X.
 \end{align*}
	 	\item [(b)] We have   $$  G(x,y) =\lim_{n\to\infty} H_{K_{n}}^{-1}1_{\{x\}}(y) $$ for all $ x,y\in X $ and all exhaustions $ (K_{n}) $ of $ X $ and the convergence is both pointwise monotone increasing and with respect to $ \|\cdot\|_{h,o} $, i.e.,   for any $ x\in X $, we have $ G(x,\cdot)\in D_{h,o} $ such that $ G(x,\cdot) $ can be characterized in $ D_{h,o} $ as follows: Let $ g\in D_{h,o}  $ and $ x\in X $. Then, the following statements are equivalent:
	 	\begin{itemize}
	 		\item [(i)] $ g =G(x,\cdot)$.
	 		\item [(ii)] $\overline{ h}(g/g(x))=\inf\{{h}(f)\mid f\in C_{c}(X), f(x)=1\} $.
	 		\item [(iii)] $ Hg=1_{x} $
	 		\item [(iv)] $\overline{ h}(g,f)=f(x) $ for all $ f\in D_{h,o} $.
 	\end{itemize}
 		 	\item [(c)]  $ G(x,\cdot) $ is the smallest positive solution $ g $ to the inequality
 \begin{align*}
 \H g \geq 1_{\{x\}},\qquad x\in X.
 \end{align*}

	 \end{itemize}
\end{thm}

\begin{proof}[Proof of Theorem~\ref{t:GF}](a) The first equality straightforwardly follows from the Laplace transform, Lemma~\ref{l:Laplace}. The monotonicity in the convergence follows from the resolvent identity
\begin{align*}
(H^{(1)}-\lm)^{-1}1_{\{x\}}\!-\!(H^{(1)}-\mu)^{-1}1_{\{x\}}\!=\!(\lm\!-\!\mu)(H^{(1)}\!-\!\lm)^{-1}(H^{(1)}\!-\!\mu)^{-1}1_{\{x\}},
\end{align*}
and since by Corollary~\ref{c:posivitity_preserving}, the resolvents are positivity preserving.
	
Letting $ g_{\lm,x}=(H^{(1)}-\lm)^{-1}1_{\{x\}} $, the symmetry follows since
	\begin{align*}
g_{\lm,x}(y)=\langle	(H^{(1)}-\lm)^{-1}1_{\{x\}},1_{\{y\}}\rangle=\langle	1_x,(H^{(1)}-\lm)^{-1}1_{\{y\}}\rangle=	g_{\lm,y}(x).
	\end{align*}
	The last statement of (a) follows from the Harnack principle (Lemma~\ref{l:convergence}).
	
\medskip
	
	(b) Let $ u_{n}=H_{K_n}^{-1}1_{\{x\}} $, $ n\in\N $.
	By the minimum principle (Lemma~\ref{l:minimum}), we deduce   $ u_{n}-u_{k}\ge 0 $ on $ K_{k} $ if $ n\ge k $ and, hence, $ (u_{n}) $ is monotone increasing. In addition,
by the minimum principle (Lemma~\ref{l:minimum}) applied to $u= G(x,\cdot) - u_n$, we have $ 0\leq u_{n}\leq G(x,\cdot)$, on $ K_{n} $, $ n\in\N $.
  Therefore, the sequence $ (u_{n}) $ has a limit $ u $ that solves $ \H u=1_{\{x\}} $,  by the Harnack principle (Lemma~\ref{l:convergence}).
  $0<u\leq G(x,\cdot)$.

Let $ v $ be a strictly positive superharmonic function which exists by the Allegretto-Piepen\-brink theorem (Theorem~\ref{t:AP}). Then by \cite[Proposition 2.6 and 2.7]{KL1}, for the operator $ H_{v}^{(v^{2})} $ and its restrictions $ H_{v,K} $ to $C_{c}(K)= \ell^{2}(K,v^{2}) $,  we have for any $\gl>0$
\begin{align*}
T^{-1}_{v}(H_{K_{n}}-\lm)^{-1}1_{\{x\}}&=(H_{v,K_{n}}-\lm)^{-1}T^{-1}_{v}1_{\{x\}}\\
&\nearrow (H_{v}^{(v^{2})}-\lm)^{-1}T^{-1}_{v}1_{\{x\}}\\
&=T^{-1}_{v}(H^{(1)}-\lm)^{-1}1_{\{x\}}
\end{align*}
as $ n\to\infty $, where the convergence is monotone increasing.
Thus,
\begin{align*}
G(x,\cdot)&=\lim_{\lm\nearrow 0}(H^{(1)}-\lm)^{-1}1_{\{x\}}=\lim_{\lm\nearrow 0}\lim_{n\to\infty}(H_{K_{n}}^{(1)}-\lm)^{-1}1_{\{x\}}=\lim_{n\to\infty}H_{K_{n}}^{-1}1_{\{x\}},
\end{align*}
where the limits interchange due to the monotone increasing limits in both parameters.

Moreover,
\begin{align*}
h(g_{n})= \langle H_{K_{n}}g_{n},g_{n}\rangle=g_{n}(x)\leq G(x,x) \leq C.
\end{align*}
Hence, $(g_{n} )$ is a bounded sequence in the Hilbert space $ (D_{h,o},\ov h) $. By the Banach-Alaoglu theorem there exists a weakly convergent subsequence of $ (g_{n}) $ whose limit
in view of (b) coincides with the pointwise limit which is $g= G(x,\cdot) $. Hence, $ G(x,\cdot)\in D_{h,o} $ and by Fatou's lemma $ \overline h(G(x,\cdot))\leq G(x,x) $. Hence by Green's formula
\begin{align*}
\overline {h}(G(x,\cdot)-g_{n})=\overline {h}(G(x,\cdot)) -2\sum_{y\in X}HG(x,y) g_{n}(y)+\sum_{y\in X}Hg_{n}(y) g_{n}(y)\leq
G(x,x)-g_{n}(x)
\end{align*}
which converges to $ 0 $ as $ n\to\infty $. Hence, $ g_{n} $ converges to $ G(x,\cdot) $ with respect to $ \|\cdot\|_{h,o} $. Moreover, this as well immediately yields $$  \overline {h}(G(x,\cdot))=G(x,x) . $$

Let us turn to the equivalence.

(i) $ \Longrightarrow $ (ii): It is elementary to see that $ g_{n}=H_{K_{n}}^{-1}1_{x}/ H_{K_{n}}^{-1}1_{x}(x)$, $ n\in\N $, minimizes the restriction $ h_{n} $ of $ h $ to $\{\ph\in C_{c}(K_{n}) \mid\ph(x)=1\} $ for an exhaustion $ (K_{n}) $ of $ X $ with finite sets and $ x\in K_{n} $. Therefore, we have for all $ f\in C_{c}(X)  $ with $ f(x)=1 $
\begin{align*}
\overline{h}(G(x,\cdot)/G(x,x))=\lim_{n\to\infty} h_{n}(g_{n})\leq \lim_{n\to\infty} h_{n}(f)=\lim_{n\to\infty} h(f).
\end{align*}

(ii) $ \Longrightarrow $ (iii): Assume $\overline g= g/g(x) $ minimizes $ \overline h $ on  $\{f\in D_{h,o} \mid f(x)=1\} $. Hence, the function $ t\mapsto \overline h (\overline g+t1_{y}) $ has derivative zero at $ t=0 $ for all  $ y\neq x $, i.e., we have
\begin{align*}
0=\frac{d}{dt}\overline h (\overline g+t1_{y}) \vert_{t=0}=2\overline h(\overline g,1_{y})=2H \overline g(y).
\end{align*}
So, does the function  $ t\mapsto \overline h ((1-t)\overline g+t1_{x}) $, i.e.,
\begin{align*}
0=\frac{d}{dt}\overline h ((1-t)\overline g+t1_{x}) \vert_{t=0}=-2\overline h(\overline g)+2\overline h(\overline g,1_{x})=2(1-Hg(x)).
\end{align*}
This proves $ Hg=1_{x} $.

(iii) $ \Longrightarrow $ (iv): This follows directly by Green's formula for all $ f\in C_{c}(X) $ and by density for all functions in $ D_{h,o} $.

(iv) $ \Longrightarrow $ (i): By Theorem~\ref{t:subcriticality}~(iv),
$ (D_{h,o},\overline h) $ is a Hilbert space. Moreover, by Lemma~\ref{lem_pointwise}~(b) the map $\delta_{x}: D_{h,o}\to\R $, $ f\mapsto f(x) $ is a continuous linear functional. Hence, by the Riesz representation theorem there is a unique $ g\in D_{h,o} $ such that $ \overline h(g,f)=\delta_{x}(f)=f(x). $ However, by Green's formula $ G(x,\cdot) $ satisfies the equation $ \overline h(G(x,\cdot),\ph)=\ph(x) $ for all $ \ph\in C_{c}(X) $. By density of $ C_{c}(X) $ in $ D_{h,o} $ we infer $ g=G(x,\cdot) $.

\medskip

(c) Let $ u $ be a $ \H $-superharmonic function such that $ \H u \geq 1_{\{x\}}$ and $ g_{n}=H_{K_{n}}^{-1}1_{\{x\}} $ for an exhaustion $ (K_{n}) $. Then, by the minimum principle we infer $ u-g_{n} \ge0$ on $ K_{n} $. By (b) we infer $ u\ge G(x,\cdot) $.
\end{proof}

\subsection{Uniform subcriticality}\label{s:USC}
In this section we turn to the topic of uniform subcriticality. In the continuum this notion goes back to work of Pinchover \cite{P88}. In the discrete setting of random walks this notion is known as uniform transience and was investigated in the joint work of Barlow, Coulhon, and Grigor'yan
\cite{BCG01}, and in the paper of Windisch \cite{Win},  and also in the works of Kasue  \cite{Kas1,Kas2} under no particular name. There is also a recent work which discusses the Dirichlet problem  for the Royden boundary on uniformly transient graphs \cite{KLSW17}.

\begin{defin}[Uniform subcriticality]  Let $b$ be a  connected graph over $X$, let $q$ be a
	potential such that $h\ge0$ on $C_{c}(X)$. We say  $ h $ is \emph{uniformly subcritical} if there is a constant $ C>0 $ such that for every $ x\in X $ there is $ w\ge0 $ such that $ w(x)\ge C $ and $ h-w\ge0 $ on $ C_{c}(X) $.
	
\end{defin}
We denote
\begin{align*}
C_{0}(X):=\ov{C_{c}}^{\|\cdot\|_{\infty}}.
\end{align*}

The following characterization is shown in the continuum setting \cite{P88}. For Laplace type operators on graphs (with non-negative potentials) a similar results are found in \cite{Kas1,Kas2,KLSW17}.

\begin{thm}[Uniform subcriticality characterization]\label{t:usubcriticality}
	Let $b$ be a  connected graph over $X$, let $o \in X$, let $q$ be a
	potential such that $h\ge0$ on $C_{c}(X)$. Then the following assertions are
	equivalent:
	\begin{itemize}
		\item [(i)] $h$ is uniformly subcritical in $X$.
		\item [(i')] There is $ C>0 $ such that
		\begin{align*}
		h(\ph)\ge C\|\ph\|_{\infty}^{2},\qquad  \ph\in C_{c}(X).
		\end{align*}
\item [(i'')] There is $ C>0 $ such that
				\begin{align*}
				\|\ph\|_{h,o} \ge C\|\ph\|_{\infty}^{2},\qquad  \ph\in C_{c}(X).
				\end{align*}
		\item [(ii)] There is $ C>0 $ such that $G(x,x)\leq C$
for all $x\in X$.

	\item [(iii)] There is $ C > 0 $ such that for every vertex $x\in X$, there exists a positive $H$-super\-harmonic function $u (x)\leq C$ such that $\H u \geq 1_{\{x\}}$ in $X$.
		\item[(iv)] $D_{h,o}\subseteq C_{0}(X) $.
		\item [(v)] There is $ C>0 $ such that  		 $\inf_{x\in X}\mathrm{cap}_{h}(\{x\})>C $ for all $x\in X$.
	\end{itemize}
In particular, $ G(x,\cdot)\in C_{0} (X)$	 for all $ x\in X $.
\end{thm}

\begin{proof} The equivalences (i) $ \Longleftrightarrow $ (i') $ \Longleftrightarrow $ (v) follow directly from the definitions.

\medskip

The implications (i') $ \Longrightarrow $ (i'') $ \Longrightarrow $ (iv) are clear.

\medskip

(iv) $ \Longrightarrow $ (i'):  By the closed graph theorem the map $ (D_{h,o},\|\cdot\|_{o})\to (C_{0}(X),\|\cdot\|_{\infty}) $ is continuous for any $ o\in X $. Hence, (i'') follows.

\medskip

(i) $ \Longrightarrow $ (ii): This follows from Lemma~\ref{l:GFbound}.

\medskip

(ii) $ \Longrightarrow $ (i): Let $ g=G(x,\cdot) $, $ x\in X $. Then, by the ground state transform, we have with $ w=1/g(x) $
\begin{align*}
h(\ph)=h_{g}\left(\frac{\ph}{g}\right)+\langle w \ph,\ph\rangle
\ge \langle w \ph,\ph\rangle.
\end{align*}
Since $ G(x,x)< C $, $ x\in X $, we have $ w(x)\ge 1/C $ and, therefore, $ h $ is uniformly subcritical.

\medskip

(iv) $ \Longrightarrow $ (iii): This follows from Theorem~\ref{t:GF}~(d).

\medskip

(iii) $ \Longrightarrow $ (ii): This follows from Theorem~\ref{t:GF}~(c).

\medskip

The ``in particular'' follows by (iv) combined with Theorem~\ref{t:GF}~(d).
\end{proof}
\section{Characterization of positive/null-criticality}\label{s:NC}
Below, we  provide a characterization of null/positive-criticality which is defined next.
\begin{defin}[Null-critical/positive-critical] Let $h$ be a quadratic form associated with the formal
	Schr\"odinger operator ${\H}$, such that $h\ge0$ on $C_{c}(X)$. The
	form $h$ is called \emph{null-critical}  (resp.,
	\emph{positive-critical}) in $X$ with respect to a positive
	potential $w$ if $h$ is critical  in $X$ and $\sum_X \psi^2 w
	=\infty$ (resp., $\sum_X \psi^2 w <\infty$), where $\psi$ is the
	ground state of $h$ in $X$.
\end{defin}
Note that the null/positive-criticality of a critical form
depends also on the weight $w$.

For any nontrivial positive function $w$ consider the seminorm
$\|\cdot\|_{w}$ by
$$\|f \|_{w}=\left(\sum_{x\in X}w(x)f(x)^{2}\right)^{1/2}.$$
Whenever, $w$ is strictly positive, then $\|\cdot\|_{w}$ is a norm.
We can close the form $h$ in $\ell^{2}(X,w)$ (cf.
Theorem~\ref{t:closable}), and associate a self-adjoint operator
$H^{(w)}$ to the closure.

We prove the statement for the form $ h-w $ instead of $ h $ because this is how it is
often used. Hence we use the form
$h-w\ge0$ to define the operator $(H-w)^{(w)}$.
Recall that by Theorem~\ref{t:char_criticality}, a critical operator admits a unique ground state (up to normalization).
\begin{thm}[Positive criticality]\label{t:char_nullcriticality}
Let $b$ be a  connected graph over $X$,  and let $q$ be a potential and $ o\in X $.
 Let $w\gneqq 0$ such that $h-w\ge0$ is critical in $X$, and
let $\psi$ be the corresponding ground state of $ h-w $ normalized at $ o $.
Then, the following statements are equivalent:
\begin{itemize}
  \item [(i)] $h-w$ is positive-critical with respect to $w$. That is $\psi\in\ell^{2}(X,w)$.
  \item [(ii)] $\psi\in D_{h,o}$.
  \item [(iii)] $\psi\in D_{h,o}\cap \ell^{2}(X,w)$.
  \item [(iv)] There is $v\in D_{h,o}\cap \ell^{2}(X,{w})$
   such that
  \begin{align*}
\frac{\overline{h}(v)}{\|v\|_{w}^2}=\inf_{\ph\in C_{c}(X)}
\frac{h(\ph)}{\|\ph\|_{w}^2}\,.
  \end{align*}
  \item [(v)]  $\psi\in D_{h,o}\cap \ell^{2}(X,{w})$
   and
  \begin{align*}
\frac{\overline{h}(\psi)}{\|\psi\|_{w}^2}=\inf_{\ph\in C_{c}(X)}
\frac{h(\ph)}{\|\ph\|_{w}^2}\,.
  \end{align*}
\end{itemize}
Furthermore, if $w$ is strictly positive, then also the following
statement is equivalent:
\begin{itemize}
  \item [(vi)] $\psi$ is an eigenfunction of the operator
  $(H-w)^{(w)}$ with eigenvalue $0$.
\end{itemize}
\end{thm}
\begin{proof}Since $h-w$ is critical, it follows that $h\ge0$, and we have
\begin{align*}
    \inf_{\ph\in C_{c}(X)}\frac{h(\ph)}{\|\ph\|_{w}^{2}}=1.
\end{align*}

Furthermore, for any null-sequence $(\ph_{n})$ in $C_{c}(X)$ with
$\ph_{n}(o)=\psi(o)$  and $
\lim_{n\to\infty}(h-w)(\ph_{n})=0$, we have by
Theorem~\ref{t:char_criticality} (iv'), $\ph_{n}\to \psi$ pointwise. This
gives the equivalence ``(i) $\Longleftrightarrow$ (ii) $\Longleftrightarrow$
(iii) $\Longleftrightarrow$ (iv) $\Longleftrightarrow$ (v)'' with $v=\psi$
in (iv).

Now, assume that $w$ is strictly positive. Then (vi)
$\Longrightarrow$ (i) is obvious because $\psi $ being an
eigenfunction implies $\psi\in\ell^{2}(X,w)$.

We finally prove (iii) $\Longrightarrow$ (vi). In order to   show $\psi\in D((H-w)^{(w)})$
it suffices to demonstrate $\psi \in D((h-w)^{(w)})$. Then $\psi\in
D((H-w)^{(w)})$ follows directly from the definition of the operator
domain using that $C_{c}(X)$ is dense in the form domain,
$(\H-w)\psi=0$ and the Green formula (Lemma~\ref{l:Green}). Since the
ground state transform $(h-w)_{\psi}$ is Markovian we have by
general Dirichlet form theory
\begin{align*}
    D((h-w)_{\psi}^{(m)})=D_{(h-w)_{\psi}, o}\cap \ell^{2}(X,m)
\end{align*}
for any measure $m$ and $o\in X$, (see \cite[Theorem~1.5.2]{Fuk} or
for the special case of graphs \cite[Lemma~1.6]{KLSW17}). Hence, by the ground state
transform we get
\begin{align*}
    D((h-w)^{(w)})&=T_{\psi^{-1}}    D((h-w)_{\psi}^{(\psi^{2}w)})\\
&    =T_{\psi^{-1}} \left(D_{(h-w)_{\psi}, o} \cap \ell^{2}(X,\psi^{2}w)\right)\\
&=
    D_{h-w,o}\cap \ell^{2}(X,w).
\end{align*}
This shows the claim.
\end{proof}
Positive criticality with respect to the weight $w=1$ is characterized also by the large time behavior of the heat kernel. For the counterpart result in the case of second-order elliptic partial differential operators, see \cite{P13} and references therein.
\begin{thm}[{\cite[Theorem~3.1 and Corollary~5.6]{KLVW}}]\label{thm_large_time}
Let $H$ be a Schr\"odinger operator on $X$, and suppose that $H\geq 0$ in $X$. Let $\gl_0:=\lm_{0}(H^{(1)})$ be the bottom of
the spectrum  of the selfadjoint operator $H^{(1)}$, and denote by $p_t(x,y):=\mathrm{e}^{-tH^{(1)}}1_{\{x\}}(y)$
the heat kernel of operator ${\H}$. Then for each  $x,y\in X$
$$\lim_{t\to\infty} \frac{\log p_t(x,y)}{t}=-\gl_0.$$
 Moreover,
 \begin{equation*}
\lim_{t\to\infty} \mathrm{e}^{\lambda_0 t}p_t(x,y)= \Psi(x)\Psi(y),
\end{equation*}
where $\Psi=0$ unless $H-\gl_0$ is positive critical in $X$, and in this case, $\Psi\in \ell^2(X)$ is the normalized ground state of $H-\gl_0$.
\end{thm}
Let $H\geq 0$ in $X$. The following theorem demonstrates that positive criticality with respect to a weight $w$ is characterized by the behavior of the positive minimal Green function of the operator $H-\gl w$ as $\gl\nearrow 0$. For the counterpart result in the case of second-order elliptic partial differential operators, see \cite[Theorem~1.1]{P92}.
\begin{thm}\label{thm_lGl}
Let $H$ be a Schr\"odinger operator on $X$, and suppose that $H\geq 0$ in $X$. Let $w$ be a positive weight on $X$. For $\gl<0$, let
$G_\gl(x,x_0)$ be the positive minimal Green function of the operator $H-\gl w$ on $X$.
Then for each  $x,x_0\in X$ we have
 \begin{equation*}
\lim_{\gl\nearrow \,0}  (-\gl) G_\gl(x,x_0)= \phi(x)\phi(x_0),
\end{equation*}
where $\phi=0$ unless $H$ is positive critical in $X$ with respect to the weight $w$, and in this case, $\phi$ is the normalized ground state of $H$ in $\ell^2(X,w)$.
\end{thm}
\begin{proof}
If $H$ subcritical, then clearly
\begin{equation*}
\lim_{\gl\nearrow \,0}  (-\gl) G_\gl(x,x_0)= 0.
\end{equation*}
So, we may assume that $H$ is critical in $X$.

Fix $x_0, x_1 \in X$.
For $s<0$, let $G_s(\cdot,x_0)$ be the minimal positive Green function of $H-sW$ in $X$
satisfying $(H-sw)G(\cdot, x_0) =\delta_{x_0}$.

Set
\[
u_s(x):= \frac{G_s(x,x_0)}{G_s(x_1, x_0)}.
\]

 \textbf{Claim}: For any $x\in X$
\begin{equation}\label{s_t0_0}
\lim_{s \nearrow 0} u_s(x)= \frac{\phi(x)}{\phi(x_1)},
\end{equation}
where $\phi$ is the unique positive ground state for $H$.


Indeed, for any fixed $t<s<0$, the function $u_s$ satisfies
$$(H-tw)u_s \geq 0 \quad \mbox{in } X,$$
as well as $u_s(x_1)=1$.
Therefore, by the Harnack principle, a standard exhaustion and diagonalization arguments, and up to a subsequence, the sequence $(u_s)$ converges to a positive supersolution $\vgf$ the equation $(H-tw)u=0$ in $X$.


Letting $t\nearrow 0$, we obtain by Fatou's Lemma that pointwise, $H\vgf\geq 0$ in $X$. So, $\vgf$ is a positive super solution of the equation $Hu=0$ in $X$ satisfying $\vgf(x_1)=1$.
The criticality of $H$ in $X$ implies now that $\vgf(x)=\phi(x)/\phi(x_1)$. In particular, the limit does not depend on the subsequence, hence,
$\lim_{s \nearrow 0} u_s(x)= \phi(x)/\phi(x_1)$, and the claim is proved.

Now, by the resolvent equation and the symmetry of the Green's function, we have for any $t<s<0$
$$G_s(x,x_0)=G_t(x,x_0)+(s-t)\, \sum_{z \in X} w(z)G_t(x,z) G_s(z, x_0) .$$
Consequently,
\begin{equation}\label{eq-res_eq}
\frac{G_s(x,x_0)}{G_s(x_1,x_0)} =\frac{G_t(x,x_0)}{G_s(x_1,x_0)}+(s-t)\,\sum_{z \in X}w(z) G_t(x,z)\frac{G_s(z,x_0)}{G_s(x_1,x_0)}.
\end{equation}
Hence, 
\begin{equation}\label{eq-res_eq7}
\frac{G_s(x,x_0)}{G_s(x_1,x_0)} \geq (s-t)\,\sum_{z \in X}w(z) G_t(x,z)\frac{G_s(z,x_0)}{G_s(x_1,x_0)}.
\end{equation}
Let $s \nearrow 0$. By the claim proven above and using Fatou's lemma, this yields
\begin{equation}\label{eq-ineq}
\frac{\phi(x)}{\phi(x_1)}\geq
-t\, \sum_{z \in X} w(z)  G_t(x,z)\frac{\phi(z)}{\phi(x_1)}.
\end{equation}
Fix $t<0$ and let
$$v_t(x):=-t\, \sum_{z \in X} w(z)  G_t(x,z)\phi(z).$$
Then
\begin{equation}\label{hvt}
    Hv_t=-tw(\phi-v_t).
\end{equation}
 By \eqref{eq-ineq} we have that
$Hv_t\geq 0$ in $X$. Moreover, \eqref{eq-ineq} implies that $v_t$ is a positive supersolution  of the equation $Hu=0$ in $X$. By the maximum principle, either $v_t$ is strictly positive or identically zero. So, $v_t>0$. But, due the criticality of $H$ in $X$, it follows that $v_t$ is a solution, and $\phi -v_t$ are either strictly positive solution of the equation $Hu=0$ in $X$, or the zero solution. Consequently, $v_t=\alpha \phi$, where $0 <\ga\leq 1$. In fact, $\ga=1$ since otherwise, \eqref{hvt} implies the $v_t$ is strictly positive supersolution. So,
\begin{equation}\label{eq-eq}
\phi(x)=
-t\, \sum_{z \in X} w(z)  G_t(x,z)\phi(z) \qquad \forall x\in X.
\end{equation}

Taking the liminf as $t \nearrow 0$ on the right hand side of \eqref{eq-eq}, an easy computation
involving again Fatou's lemma,  shows that one obtains a positive supersolution of the equation $Hu=0$ in $X$.
Consequently, by uniqueness of those supersolutions, we arrive at
\begin{equation}\label{1}
\phi(x)= \lim_{t \nearrow 0} \left( - t\, \sum_{z \in X} G_t(x,z)w(z)\phi(z) \right).
\end{equation}
Hence,
$$\phi(x)=\lim_{t \nearrow 0}\left(-tG_t(x,x_0) \sum_{z \in X} \frac{G_t(x,z)}{G_t(x,x_0)}w(z)\phi(z) \right),$$
which in turn gives
$$\limsup_{t \nearrow 0}\left(-tG_t(x,x_0)\right)=\limsup_{t \nearrow 0}\left(\frac{\phi(x)}{\sum_z \frac{G_t(x,z)}{G_t(x,x_0)}w(z)\phi(z)}\right). $$
Therefore, Fatou's lemma yields
\begin{equation}\label{tgtleqfi}
\limsup_{t \nearrow 0}\left(-tG_t(x,x_0)\right)\leq \frac{\phi(x)\phi(x_0)}{\sum_z \phi^2(z)w(z)}.
\end{equation}
In particular, if $H$ is null-critical with respect to $w$, then $\limsup_{t \nearrow 0}\left(-tG_t(x,x_0)\right)=0$.

\medskip

It remains to deal with the positive-critical case.

Note that by \eqref{tgtleqfi}, the sequence $(-tG_t(x,x_0))$ is locally bounded, and therefore, as above,  and up to a subsequence, it converges to a positive supersolution of the equation $Hu=0$ in $X$. Recall that the ground state $\gf$ is the unique supersolution of the equation $Hu=0$ in $X$. Consequently, for any subsequence $t_j \nearrow 0$ there exists $\gb\geq 0$ such that
$$\lim_{t_j \nearrow 0}\left(-t_jG_{t_j}(x,x_0)\right) =\gb \frac{\phi(x)\phi(x_0)}{\sum_z \phi^2(z)w(z)}.$$
We observe from inequality~\eqref{tgtleqfi} that $\gb \leq 1$. We still need to show that $\gb \geq 1$.
For this purpose, note first that the inequality~\eqref{tgtleqfi} holds true for all $x_0 = z \in X$.
Combining this with inequality~\eqref{1} and with Fatou's lemma, one shows that
$$\phi(x_1) \leq \sum_{z \in X} \lim_{j\to \infty}(-t_j)G_{t_j}(x_1,z)w(z)\phi(z) \leq  \gb \frac{\phi(x_1)\sum_z \phi^2(z)w(z)}{\sum_z \phi^2(z)w(z)}.$$
So, $\gb \geq 1$. Since the subsequence $(t_j)$ was chosen arbitrarily, the proof of the theorem is finished.

\end{proof}
\begin{rem}
Using \eqref{eq-res_eq7} and the Martin boundary approach it follows that for any positive potential $w$, inequality \eqref{eq-ineq} holds even if $H$ is subcritical in $X$ and $\phi$ is {\em any} positive solution of the equation $Hu=0$ in $X$. Moreover, as it is shown above, in the critical case equality holds true for any $x\in X$ (see \eqref{eq-eq}).
\end{rem}
\begin{rem}
For the weight $w=1$, Theorem~\ref{thm_lGl} follows directly from Theorem~\ref{thm_large_time} using a classical Abelian theorem \cite[Theorem~10.2]{Sim}.
\end{rem}
\bigskip
\footnotesize
\noindent\textit{Acknowledgments.}
 M.~K. is grateful to the Department of Mathematics at the Technion for the hospitality during his visits and acknowledges the financial support of the German Science Foundation. Y.~P. and F.~P. acknowledge the support of the Israel Science Foundation (grants No. 970/15) founded by the Israel Academy of Sciences and Humanities. F.~P. is grateful for support through
a Technion Fine Fellowship.


\begin{thebibliography}{SK}
\bibitem{BCG01}
M.~Barlow, T.~Coulhon and A.~Grigor'yan,
\newblock Manifolds and graphs with slow heat kernel decay.
\newblock {\em Invent. Math.}, \textbf{144} (2001), 609--649.

\bibitem{BP}
S. Beckus and Y. Pinchover,
\newblock Shnol-type theorem for the Agmon ground state. arXiv: 1706.04869.


\bibitem{BG}
M.~Bonnefont and S.~Gol{\'e}nia,
\newblock Essential spectrum and {W}eyl asymptotics for discrete {L}aplacians.
\newblock {\em Ann. Fac. Sci. Toulouse Math. (6)}, \textbf{24} (2015), 563--624.

\bibitem{Do84}
J.~Dodziuk,
\newblock Difference equations, isoperimetric inequality and transience of
  certain random walks.
\newblock {\em Trans. Amer. Math. Soc.}, \textbf{284}(1984), 787--794.

\bibitem{FLW}
R.~L.~Frank, D.~Lenz and D.~Wingert,
\newblock Intrinsic metrics for non-local symmetric {D}irichlet forms and
  applications to spectral theory.
\newblock {\em J. Funct. Anal.}, \textbf{266} (2014), 4765--4808.

\bibitem{FS08}
R.~L.~Frank and R.~Seiringer,
\newblock Non-linear ground state representations and sharp {H}ardy
  inequalities.
\newblock {\em J. Funct. Anal.}, \textbf{255} (2008), 3407--3430.

\bibitem{FSW08}
R.~L.~Frank, B.~Simon and T.~Weidl,
\newblock Eigenvalue bounds for perturbations of {S}chr\"odinger operators and
  {Ja}cobi matrices with regular ground states.
\newblock {\em Comm. Math. Phys.}, \textbf{282} (2008), 199--208.

\bibitem{Fuk}
M.~Fukushima, Y.~Oshima and M.~Takeda,
\newblock {\em Dirichlet Forms and Symmetric {M}arkov Processes}.
  de Gruyter Studies in Mathematics 19, extended edition,
\newblock Walter de Gruyter \& Co., Berlin 2011.

\bibitem{GKS}
B.~G{\"u}neysu, M.~Keller and M.~Schmidt,
\newblock A {F}eynman-{K}ac-{I}t\^o formula for magnetic {S}chr\"odinger
  operators on graphs.
\newblock {\em Probab. Theory Related Fields}, \textbf{165} (2016), 365--399.

\bibitem{HK}
S.~Haeseler and M.~Keller,
\newblock Generalized solutions and spectrum for {D}irichlet forms on graphs.
\newblock In {\em Random walks, boundaries and spectra},
  Progr. Probab. 64,  Birkh\"auser/Springer Basel AG, Basel, 2011, 181--199.

\bibitem{HKLW}
S.~Haeseler, M.~Keller, D.~Lenz and R.~Wojciechowski,
\newblock Laplacians on infinite graphs: {D}irichlet and {N}eumann boundary
  conditions.
\newblock {\em J. Spectr. Theory}, \textbf{2} (2012), 397--432.

\bibitem{HKW}
S.~Haeseler, M.~Keller, and R.~K.~Wojciechowski,
\newblock Volume growth and bounds for the essential spectrum for {D}irichlet
  forms.
\newblock {\em J. Lond. Math. Soc. (2)}, \textbf{88} (2013), 883--898.

\bibitem{Kas1}
A.~Kasue,
\newblock Convergence of metric graphs and energy forms.
\newblock {\em Rev. Mat. Iberoam.}, \textbf{26} (2010), 367--448.

\bibitem{Kas2}
A.~Kasue,
\newblock Random walks and {K}uramochi boundaries of infinite networks.
\newblock {\em Osaka J. Math.}, \textbf{50} (2013), 31--51.

\bibitem{K5}
M.~Keller,
\newblock Applications of Operator Theory - Discrete Operators.
\newblock {\em Lecture notes}, 2012.

\bibitem{KL2}
M.~Keller and D.~Lenz,
\newblock Unbounded {L}aplacians on graphs: basic spectral properties and the
  heat equation.
\newblock {\em Math. Model. Nat. Phenom.}, \textbf{5} (2010), 198--224.

\bibitem{KL1}
M.~Keller and D.~Lenz,
\newblock Dirichlet forms and stochastic completeness of graphs and subgraphs.
\newblock {\em J. Reine Angew. Math.}, \textbf{666} (2012), 189--223.

\bibitem{KLSW17}
M.~Keller, D.~Lenz, M.~Schmidt and R.~K.~Wojciechowski,
\newblock Note on uniformly transient graphs. to appear in
\newblock {\em Rev. Mat. Iberoam.}, 2017.

\bibitem{KLVW}
M.~Keller, D.~Lenz, H.~Vogt and R.~Wojciechowski,
\newblock Note on basic features of large time behaviour of heat kernels.
\newblock {\em J. Reine Angew. Math.}, \textbf{708} (2015), 73--95.

\bibitem{KePiPo2}
M.~Keller, Y.~Pinchover and F.~Pogorzelski,
\newblock Optimal Hardy inequalities for Schr\"odinger operators on graphs.
\newblock arXiv: 1612.04051.

\bibitem{KL}
H.~Kova{\v{r}}{\'{\i}}k and A.~Laptev,
\newblock Hardy inequalities for {R}obin {L}aplacians.
\newblock {\em J. Funct. Anal.}, \textbf{262} (2012), 4972--4985.

\bibitem{LSV} D.~Lenz, and P.~Stollmann, Peter and I.~Veseli\'c,
\newblock The {A}llegretto-{P}iepenbrink theorem for strongly local
{D}irichlet forms \newblock {\em Doc. Math.} \textbf{14} (2009), 167--189.

\bibitem{MoharWoess89}
B.~Mohar and W.~Woess,
\newblock A survey on spectra of infinite graphs.
\newblock {\em Bull. London Math. Soc.}, \textbf{21} (1989), 209--234.

\bibitem{M86}
M.~Murata,
\newblock Structure of positive solutions to {$(-\Delta+V)u=0$} in {${\bf
  R}^n$}.
\newblock {\em Duke Math. J.}, \textbf{53} (1986), 869--943.

\bibitem{P88}
Y.~Pinchover,
\newblock On positive solutions of second-order elliptic equations, stability
  results, and classification.
\newblock {\em Duke Math. J.}, \textbf{57} (1988), 955--980.

\bibitem{P92}
Yehuda Pinchover.
\newblock Large time behavior of the heat kernel and the behavior of the
  {G}reen function near criticality for nonsymmetric elliptic operators.
\newblock {\em J. Funct. Anal.}, 104(1):54--70, 1992.



\bibitem{P07}
Y.~Pinchover,
\newblock Topics in the theory of positive solutions of second-order elliptic
  and parabolic partial differential equations.
\newblock In {\em Spectral theory and mathematical physics: a {F}estschrift in
  honor of {B}arry {S}imon's 60th birthday}, Proc. Sympos.
  Pure Math. 76,  Amer. Math. Soc., Providence, RI, 2007, 329--355.

\bibitem{P13}
Y.~Pinchover,
\newblock Some aspects of large time behavior of the heat kernel: an overview
  with perspectives.
\newblock In {\em Mathematical physics, spectral theory and stochastic
  analysis}, Oper. Theory Adv. Appl. 232,
  Birkh\"auser/Springer Basel AG, Basel, 2013, 299--339.

\bibitem{PP}
Y.~Pinchover and G.~Psaradakis,
\newblock On positive solutions of the {$(p,A)$}-{L}aplacian with a potential
  in morrey space.
\newblock Anal.~PDE {\bf 9} (2016), 1317--1358.

\bibitem{PT}
Y.~Pinchover and K.~Tintarev,
\newblock On positive solutions of {$p$}-{L}aplacian-type equations.
\newblock In {\em Analysis, partial differential equations and applications},
  Oper. Theory Adv. Appl. 193, Birkh\"auser Verlag, Basel, 2009, 245--267.

\bibitem{Pins95}
R.~G.~Pinsky,
\newblock {\em Positive Harmonic Functions and Diffusion}.
  Cambridge Studies in Advanced Mathematics 45,
\newblock Cambridge University Press, Cambridge 1995.

\bibitem{Pru64}
W.~E.~Pruitt,
\newblock Eigenvalues of non-negative matrices.
\newblock {\em Ann. Math. Statist.} \textbf{35} (1964), 1797--1800.

\bibitem{RS4}
M.~Reed and B.~Simon,
\newblock {\em Methods of Modern Mathematical Physics. {IV}. {A}nalysis of
  Operators}.
\newblock Academic Press [Harcourt Brace Jovanovich, Publishers], New
  York-London, 1978.


\bibitem{Sim} B.~Simon, \newblock {\em Functional Integration and Quantum Physics},
\newblock Pure and Applied Mathematics, 86, Academic Press, Inc., New York-London, 1979.


\bibitem{Si0}
B.~Simon,
\newblock Brownian motion, {$L^{p}$} properties of {S}chr\"odinger operators
  and the localization of binding.
\newblock {\em J. Funct. Anal.} \textbf{35} (1980), 215--229.

\bibitem{Tak14}
M.~Takeda,
\newblock Criticality and subcriticality of generalized {S}chr\"odinger forms.
\newblock {\em Illinois J. Math.} \textbf{58} (2014), 251--277.

\bibitem{Tak16}
M.~Takeda,
\newblock Criticality for {S}chr\"odinger type operators based on recurrent
  symmetric stable processes.
\newblock {\em Trans. Amer. Math. Soc.} \textbf{368} (2016), 149--167.

\bibitem{TH}
N.~Torki-Hamza,
\newblock Laplaciens de graphes infinis ({I}-graphes) m\'etriquement complets.
\newblock {\em Confluentes Math.} \textbf{2} (2010), 333--350.

\bibitem{Vee63}
W.~Veech,
\newblock The necessity of {H}arris' condition for the existence of a
  stationary measure.
\newblock {\em Proc. Amer. Math. Soc.} \textbf{14} (1963), 856--860.

\bibitem{VeJo67}
D.~Vere-Jones,
\newblock Ergodic properties of nonnegative matrices. {I}.
\newblock {\em Pacific J. Math.} \textbf{22} (1967), 361--386.

\bibitem{Win}
D.~Windisch,
\newblock Entropy of random walk range on uniformly transient and on uniformly
  recurrent graphs.
\newblock {\em Electron. J. Probab.} \textbf{15} (2010), 1143--1160.

\bibitem{Woe-Book}
W.~Woess,
\newblock {\em Random Walks on Infinite Graphs and Groups},
  Cambridge Tracts in Mathematics 138,
\newblock Cambridge University Press, Cambridge, 2000.

\bibitem{Woj1}
R.~K.~Wojciechowski,
\newblock {\em Stochastic completeness of graphs}.
\newblock Thesis (Ph.D.)--City University of New York, ProQuest LLC, Ann Arbor, MI. 2008.

\end{thebibliography}

\end{document}